\documentclass[11pt]{amsart}
\usepackage{amsmath, amsfonts, eucal, amssymb, upref, amsbsy, amsthm}

\newtheorem{thm}{Theorem}[section]
\newtheorem{lmm}[thm]{Lemma}
\newtheorem{cor}[thm]{Corollary}
\newtheorem{prop}[thm]{Proposition}
\newtheorem{defn}[thm]{Definition}

\newcommand{\cc}{\mathbb{C}}
\newcommand{\cov}{\mathrm{Cov}}

\newcommand{\ee}{\mathbb{E}}

\newcommand{\ii}{\mathbb{I}}

\newcommand{\mi}{\mathcal{I}}
\newcommand{\ml}{\mathcal{L}}

\newcommand{\pp}{\mathbb{P}}

\newcommand{\ra}{\rightarrow}
\newcommand{\rr}{\mathbb{R}}

\newcommand{\st}{\sqrt{t}}
\newcommand{\sst}{\sqrt{1-t}}
\newcommand{\tr}{\operatorname{Tr}}

\newcommand{\var}{\mathrm{Var}}

\newcommand{\vp}{\varphi^\prime}

\newcommand{\fpar}[2]{\frac{\partial #1}{\partial #2}}
\newcommand{\spar}[2]{\frac{\partial^2 #1}{\partial #2^2}}
\newcommand{\mpar}[3]{\frac{\partial^2 #1}{\partial #2 \partial #3}}

\newcommand{\real}{\operatorname{Re}}
\newcommand{\imag}{\operatorname{Im}}

\begin{document}
\title[Fluctuations of eigenvalues]{Fluctuations of eigenvalues and second order Poincar\'e inequalities}
\author{Sourav Chatterjee}
\address{\newline367 Evans Hall \#3860\newline
Department of Statistics\newline
University of California at Berkeley\newline
Berkeley, CA 94720-3860\newline
{\it E-mail: \tt sourav@stat.berkeley.edu}\newline 
{\it URL: \tt http://www.stat.berkeley.edu/$\sim$sourav}
}
\subjclass[2000]{60F05, 15A52}
\keywords{Central limit theorem, random matrices, linear statistics of eigenvalues, Poincar\'e inequality, Wigner matrix, Wishart matrix, Toeplitz matrix}
\thanks{The author's research was partially supported by NSF grant DMS-0707054 and a Sloan Research Fellowship.}
\maketitle

\begin{abstract}
Linear statistics of eigenvalues in many familiar classes of random matrices are known to obey gaussian central limit theorems. The proofs of such results are usually rather difficult, involving hard computations specific to the model in question. In this article we attempt to formulate a unified technique for deriving such results via relatively soft arguments. In the process, we introduce a notion of `second order Poincar\'e inequalities':  just as ordinary Poincar\'e inequalities give variance bounds, second order Poincar\'e inequalities give central limit theorems. The proof of the main result employs Stein's method of normal approximation. A number of examples are worked out, some of which are new. One of the new results is a CLT for the spectrum of gaussian Toeplitz matrices. 
\end{abstract}



\section{Introduction}\label{intro}
Suppose $A_n$ is an $n\times n$ matrix with real or complex entries and eigenvalues $\lambda_1,\ldots,\lambda_n$, repeated by multiplicities. A linear statistic of the eigenvalues of $A_n$ is a function of the form $\sum_{i=1}^n f(\lambda_i)$, where $f$ is some fixed function.
Central limit theorems for linear statistics of eigenvalues of large dimensional random matrices have received considerable attention in recent years. A very curious feature that makes these results unusual and interesting is that they usually {\it do not require normalization}, i.e.\ one does not have to divide by~$\sqrt{n}$; only centering is enough. Moreover, they have important applications in statistics and other applied areas (see e.g.\ the recent survey by Johnstone~\cite{johnstone06}).

The literature around the topic is quite large. To the best of our knowledge, the investigation of central limit theorems for linear statistics of eigenvalues of large dimensional random matrices began with the work of Jonsson~\cite{jonsson82} on Wishart matrices. The key idea is to express $\sum \lambda_i^k$ as
\[
\sum \lambda_i^k = \tr(A_n^k) = \sum_{i_1,i_2,\ldots,i_k} a_{i_1i_2}a_{i_2i_3}\cdots a_{i_{k-1}i_k}a_{i_ki_1},
\]
where $A_n$ is an $n \times n$ Wishart matrix, and then apply the method of moments to show that this is gaussian in the large $n$ limit. In fact, Jonsson proves the joint convergence of the law of $(\tr (A_n), \tr (A_n^2),\ldots,\tr(A_n^p))$ to a multivariate normal distribution (where $p$ is fixed). 

A similar study for Wigner matrices was carried out by Sina\u\i \ and Soshnikov \cite{sinaisoshnikov98, sinaisoshnikov98b}. A deep and difficult aspect of the Sina\u\i -Soshnikov results is that they get central limit theorems for $\tr(A_n^{p_n})$, where $p_n$ is allowed to grow at the rate $o(n^{2/3})$, instead of remaining fixed. They also get CLTs for $\tr(f(A_n))$ for analytic $f$.


Incidentally, for gaussian Wigner matrices, the best available results are due to Johansson~\cite{johansson98}, who characterized a large (but not exhaustive) class of functions for which the CLT holds. In fact, Johansson proved a general result for linear statistics of eigenvalues of random matrices whose entries have a joint density with respect to Lebesgue measure of the form
$Z_n^{-1}\exp(-n\tr V(A))$, 
where $V$ is a polynomial function and $Z_n$ is the normalizing constant. These models are widely studied in the physics literature. Johansson's proof relies on a delicate analysis of the joint density of the eigenvalues, which is explicitly known for this class of matrices.

Another important contribution is the work of Diaconis and Evans \cite{diaconisevans01}, who proved similar results for random unitary matrices. Again, the basic approach relies on the method of moments, but the computations require new ideas because of the lack of independence between the matrix entries. However, as shown in \cite{diaconisshahshahani94, diaconisevans01}, strikingly exact computations are possible in this case by  invoking some deep connections between symmetric function theory and the unitary group.


An alternative approach, based on Stieltjes transforms, has been developed in Bai and Yao \cite{baiyao03} and Bai and Silverstein \cite{baisilverstein04}. This approach has its roots in the semi-rigorous works of Girko \cite{girko90} and Khorunzhy, Khoruzhenko, and Pastur \cite{kkp96}.

Yet another line of attack, via stochastic calculus, was initiated in the work of Cabanal-Duvillard \cite{cabanalduvillard01}. The ideas were used by Guionnet \cite{guionnet02} to prove central limit theorems for certain band matrix models. Far reaching results for a very general class of band matrix models were later obtained using combinatorial techniques by Anderson and Zeitouni \cite{andersonzeitouni06a}.

Other influential ideas, sometimes at varying levels of rigor, come from the papers of Costin and Lebowitz \cite{costinlebowitz95}, Boutet de Monvel, Pastur and Shcherbina \cite{boutet95}, Johansson \cite{johansson97}, Keating and Snaith \cite{keatingsnaith00}, Hughes et.~al.~\cite{hughesetal00}, Soshnikov \cite{soshnikov02}, Israelson \cite{israelson01} and Wieand~\cite{wieand02}. The recent works of Anderson and Zeitouni \cite{andersonzeitouni06}, Dumitriu and Edelman \cite{dumitriuedelman06}, Rider and Silverstein \cite{ridersilverstein06}, Rider and Vir\'ag~\cite{ridervirag06}, Jiang \cite{jiang06}, and Hachem et.\ al.\ \cite{hachem06, hachem07} provide several illuminating insights and new results. The recent advances in the theory of second order freeness (introduced by Mingo and Speicher \cite{mingospeicher06}) are also of great interest.


In this paper we introduce a result (Theorem \ref{matrix}) that may provide a unified `soft tool' for matrices that can be easily expressed as smooth functions of independent random variables. The tool is soft in the sense that we only need to calculate various upper and lower bounds rather than perform exact computations of limits as required for existing methods. 
(In this context, it should be noted that soft arguments are possible even in the combinatorial techniques, if one works with cumulants instead of moments, e.g.\ as in \cite{andersonzeitouni06a}, Lemma 4.10). 

We demonstrate the scope of our approach with applications to generalized Wigner matrices, gaussian matrices with arbitrary correlation structure, gaussian Toeplitz matrices, Wishart matrices, and double Wishart matrices. 

\subsection{The intuitive idea}
Let us now briefly describe the main idea. Suppose $X = (X_1,\ldots,X_n)$ is a vector of independent standard gaussian random variables, and $g:\rr^n \ra \rr$ is a smooth function. Let $\nabla g$ denote the gradient of $g$. We know that if $\|\nabla g(X)\|$ is typically small, then $g(X)$ has small fluctuations. In fact, the gaussian Poincar\'e inequality says that
\begin{equation}\label{poincareineq}
\var(g(X)) \le \ee\|\nabla g(X)\|^2.
\end{equation}
Thus, the size of $\nabla g$ controls the variance of $g(X)$. Based on this, consider the following speculation: Is it possible to extend the Poincar\'e inequality to the `second order', as a method of determining whether $g(X)$ is approximately gaussian by inspecting the behavior of the second order derivatives~of~$g$? 


The speculation turns out to be correct (and useful for random matrices), although in a rather mysterious way. The following example is representative of a general phenomenon.

Suppose $B$ is a fixed $n \times n$ real symmetric matrix, and the function $g:\rr^n \ra \rr$ is defined as 
\[
g(x) = x^t Bx,
\]
where $x^t$ denotes the transpose of the vector $x$.
Let $X = (X_1,\ldots,X_n)$ be a vector of independent standard gaussian random variables, and let us ask the question ``When is $g(X)$ approximately gaussian?''. 

Now, if $\lambda_1, \lambda_2, \ldots, \lambda_n$ are the eigenvalues of $B$ with corresponding eigenvectors $u_1,u_2,\ldots,u_n$, then 
\[
g(X) = \sum_{i=1}^n \lambda_i Y_i^2,
\]
where $Y_i = u_i^t X$. Since we can assume without loss of generality that $u_1,\ldots,u_n$ are mutually orthogonal, therefore $Y_1,\ldots, Y_n$ are again i.i.d.\ standard gaussian. This seems to suggest that $g(X)$ is approximately gaussian if and only if `no eigenvalue dominates in the sum'. In fact, one can show that $g(X)$ is approximately gaussian if and only if
\[
\max_i |\lambda_i|^2 \ll \sum_i \lambda_i^2.
\]
Now $\nabla^2 g (x) \equiv 2B$, where $\nabla^2 g$ denotes the Hessian matrix of $g$. Thus, the question about the gaussianity of $g(X)$ can be reduced to a question about the negligibility of the operator norm squared of $\nabla^2 g(X)$ $( = 2\max |\lambda_i|^2)$ in comparison to the variance of $g(X)$ $( = 2\sum \lambda_i^2)$.

In Theorem~\ref{poincare1} we generalize this notion to show that for any smooth $g$, $g(X)$ is approximately gaussian whenever {\it the typical size of the operator norm squared of $\nabla^2g(X)$ is small compared to $\var (g(X))$}, and a few other conditions are satisfied. An outline of the rigorous proof is given in the next subsection.

The idea is applied to random matrices as follows. We consider random matrices that can be easily expressed as functions of independent random variables, and think of the linear statistics of eigenvalues as functions of these independent variables. The setup can be pictorially represented as
\[
\text{large vector } X \ra \text{ matrix } A(X) \ra \text{ linear statistic ${\textstyle\sum_i} f(\lambda_i)$} =: g(X).
\]
The main challenge is to evaluate the second order partial derivatives of $g$. However, our task is simplified (and the argument is `soft') because we only need bounds and not exact computations. Still, a considerable amount of bookkeeping is involved. We  provide a `finished product' in Theorem \ref{matrix} for the convenience of potential future users of the method. 

A discrete version of this idea is investigated in the author's earlier paper~\cite{chatterjeenormal1}. However, no familiarity with~\cite{chatterjeenormal1} is required here.

\subsection{Outline of the proof via Stein's method} The argument for general $g$ is not as intuitive as for quadratic forms. It begins with Stein's method~\cite{stein72, stein86}: If a random variable $W$ satisfies $\ee(\varphi(W)W) \approx \ee(\varphi'(W))$ for a large class of functions $\varphi$, then $W$ is approximately standard gaussian. The idea stems from the fact that if $W$ is {\it exactly} standard gaussian, then $\ee(\varphi(W)W) = \ee(\varphi'(W))$ for all absolutely continuous $\varphi$ for which both sides are well defined.
Stein's lemma (Lemma~\ref{steins} in this paper) makes this precise with error bounds.

Now suppose we are given a random variable $W$, and there is a function $h$ such that for all a.c.\ $\varphi$,
\begin{equation}\label{hdef}
\ee(\varphi(W)W) = \ee(\varphi'(W) h(W)).
\end{equation}
For example, if $W$ has a density $\rho$ with respect to Lebesgue measure, and $\ee(W) = 0$, $\ee(W^2) = 1$, then the function
\[
h(x) = \frac{\int_x^\infty y \rho(y) dy}{\rho(x)}
\]
serves the purpose. Now if $h(W) \approx 1$ in a probabilistic sense, then we can conclude that
\[
\ee(\varphi(W)W) \approx \ee(\varphi'(W)),
\]
and it would follow by Stein's method that $W$ is approximately standard gaussian. 
This idea already occurs in the literature on normal approximation~\cite{cacoullosetal94}. However, it is not at all clear how one can infer facts about $h(W)$ when $W$ is an immensely complex object like a linear statistic of eigenvalues of a Wigner matrix. One of the main contributions of this paper is an explicit formula for $h(W)$ when $W$ can be expressed as a differentiable function of a collection of independent gaussian random variables.
\begin{lmm}\label{basiclmm}
Suppose $X = (X_1,\ldots,X_n)$ is a vector of independent standard gaussian random variables, and $g:\rr^n \ra \rr$ is an absolutely continuous function. Let $W = g(X)$, and suppose that $\ee(W) = 0$ and $\ee(W^2) = 1$. Suppose $h$ is a function satisfying \eqref{hdef} for all Lipschitz $\varphi$. Then $h(W) = \ee(T(X)|W)$, where
\[
T(x) := \int_0^1 \frac{1}{2\st} \ee\biggl(\sum_{i=1}^n \fpar{g}{x_i}(x) \fpar{g}{x_i}(\st x + \sst X) \biggr) dt.
\]
\end{lmm}
\noindent Barring the technical details, the proof of this lemma is surprisingly simple. To establish \eqref{hdef}, we only have to show that for all Lipschitz $\varphi$,
\[
\ee(\varphi(W)W) = \ee(\varphi'(W) T(X)).
\]
This is achieved via gaussian interpolation. Let $X'$ be an independent copy of $X$, and let $W_t = g(\st X + \sst X')$. Since $\ee(W) = 0$, we have
\begin{align*}
\ee(&\varphi(W)W) = \ee(\varphi(W)(W_1 - W_0)) = \int_0^1 \ee\biggl(\varphi(W) \frac{\partial W_t}{\partial t}\biggr) dt\\
&= \ee\biggl(\int_0^1 \varphi(W) \sum_{i=1}^n \biggl(\frac{X_i}{2\st} - \frac{X'_i}{2\sst}\biggr)\fpar{g}{x_i}(\st X + \sst X') dt\biggr).
\end{align*}
Integration by parts on the right hand side gives the desired result. The details of the proof are contained in the proof of the more elaborate Lemma~\ref{gaussian} in Section \ref{proofs}.

Since $\ee(W^2) = 1$, taking $\varphi(x) = x$ it follows that $\ee(h(W)) = 1$. Combining this with the fact that $\var(h(W)) \le \var(T(X))$, we see that we only have to bound $\var(T(X))$ to show that $W$ is approximately gaussian. Now, if $g$ is a complicated function, $T$ is even more complicated. Hence, we cannot expect to evaluate $\var(T(X))$. On the other hand, we can always use the gaussian Poincar\'e inequality \eqref{poincareineq} to compute a bound on $\var(T(X))$. This involves working with $\nabla T$. Since $T$ already involves the first order derivatives of $g$, $\nabla T$ brings the second order derivatives into the picture. This is how we relate the smallness of the Hessian of $g$ to the approximate gaussianity of $g(X)$, leading to Theorem \ref{poincare1} in the next section.

We should mention here that a problem with Lemma \ref{basiclmm} is that we have to  know how to center and scale $W$ so that $\ee(W) = 0$ and $\ee(W^2) =1$. This may not be easy in practice.

It is also worth noting that Lemma \ref{basiclmm} can, in fact, be used to prove the gaussian Poincar\'e inequality \eqref{poincareineq} --- just by taking $\varphi(x) = x$ and applying the Cauchy-Schwarz inequality to bound the terms inside the integral in the expression for $\ee(T)$. In this sense, one can view Lemma \ref{basiclmm} as a generalization of the gaussian Poincar\'e inequality. Incidentally, the first proof of the gaussian Poincar\'e inequality in the probability literature is due to H.~Chernoff \cite{chernoff81} who used Hermite polynomial expansions. However, such inequalities have been known to analysts for a long time under the name of `Hardy inequalities with weights' (see e.g.\ Muckenhoupt \cite{muckenhoupt72}).

We should also mention two other concepts from the existing literature that may be related to this work. The first is the notion of the `zerobias transform' of $W$, as defined by Goldstein and Reinert \cite{goldsteinreinert97}. A random variable $W^*$ is called a zerobias transform of $W$ if for all $\varphi$, we have
\[
\ee(\varphi(W)W) = \ee(\varphi'(W^*)).
\]
A little consideration shows that our function $h$ is just the density of the law of $W^*$ with respect to the law of $W$ when the laws are absolutely continuous with respect to each other. However, while it is quite difficult to construct zerobias transforms (not known at present for linear statistics of eigenvalues), Lemma \ref{basiclmm} gives a direct formula for $h$.

The second related idea is the work of Borovkov and Utev \cite{borovkovutev83} which says that if a random variable $W$ with $\ee(W) = 0$ and $\ee(W^2) = 1$ satisfies a Poincar\'e inequality with Poincar\'e constant close to $1$, then $W$ is approximately standard gaussian (if the Poincar\'e constant is exactly $1$, the $W$ is exactly standard gaussian). As shown by Chen \cite{chen88}, this fact can be used to prove central limit theorems in ways that are closely related to Stein's method. Although it seems plausible, we could not detect any apparent relationship between this concept and our method of extending Poincar\'e inequalities to the second order.

\section{Second order Poincar\'e inequalities}\label{general}
All our results are for functions of random variables belonging to the following class of distributions. 
\begin{defn}
For each $c_1,c_2 > 0$, let $\ml(c_1, c_2)$ be the class of probability measures on $\rr$ that arise as laws random variables like $u(Z)$, where $Z$ is a standard gaussian r.v.\ and $u$ is a twice continuously differentiable function such that for all $x\in \rr$
\[
|u'(x)|\le c_1 \ \text{and} \ |u''(x)|\le c_2.
\]
\end{defn}
\noindent For example, the standard gaussian law is in $\ml(1,0)$. Again, taking $u = $ the gaussian cumulative distribution function, we see that the uniform distribution on the unit interval is in $\ml((2\pi)^{-1/2}, (2\pi e)^{-1/2})$. 
For simplicity, we just say that a random variable $X$ is ``in $\ml(c_1,c_2)$'' instead of the more elaborate statement that ``the distribution of $X$ belongs to~$\ml(c_1,c_2)$''.


Recall that for any two random variables $X$ and $Y$, the supremum of $|\pp(X\in B) - \pp(Y\in B)|$ as $B$ ranges over all Borel sets is called the total variation distance between the laws of $X$ and $Y$, often denoted simply by $d_{TV}(X,Y)$. Note that the total variation distance remains unchanged under any transformation like $(X,Y) \ra (f(X),f(Y))$ where $f$ is a measurable bijective map. 
Next, recall that the operator norm of an $m\times n$ real or complex matrix $A$ is defined as
\[
\|A\| := \sup\{\|Ax\|: x\in \cc^n, \|x\|=1\}.
\]
Recall that $\|A\|^2$ is the largest eigenvalue of $A^*A$. If $A$ is a hermitian matrix, $\|A\|$ is just the spectral radius (i.e.\ the eigenvalue with the largest absolute value) of $A$. This is the default norm for matrices in this paper, although occasionally we use the Hilbert-Schmidt norm
\[
\|A\|_{HS} := \bigl(\sum_{i,j} |a_{ij}|^2\bigr)^{1/2}.
\]
The following theorem gives normal approximation bounds for general smooth functions of independent random variables whose laws are in $\ml(c_1,c_2)$ for some finite $c_1, c_2$. 
\begin{thm}\label{poincare1}
Let $X = (X_1,\ldots,X_n)$ be a vector of independent random variables in $\ml(c_1,c_2)$ for some finite $c_1,c_2$. Take any $g\in C^2(\rr^n) $ and let $\nabla g$ and $\nabla^2 g$ denote the gradient and Hessian of $g$. Let 
\begin{align*}
\kappa_0 &= \biggl(\ee \sum_{i=1}^n \biggl|\fpar{g}{x_i}(X)\biggr|^4\biggr)^{1/2},\\
\kappa_1 &= (\ee\|\nabla g(X)\|^4)^{1/4}, \text{ and}\\
\kappa_2 &= (\ee\|\nabla^2 g(X)\|^4)^{1/4}.
\end{align*}
Suppose $W=g(X)$ has a finite fourth moment and let $\sigma^2 = \var(W)$. Let $Z$ be a normal random variable having the same mean and variance as $W$. 
Then
\[
d_{TV}(W,Z) \le \frac{2\sqrt{5}(c_1c_2 \kappa_0 + c_1^3\kappa_1\kappa_2)}{\sigma^2}.
\]
If we slightly change the setup by assuming that $X$ is a gaussian random vector with mean $0$ and covariance matrix $\Sigma$, keeping all other notation the same, then the corresponding bound is
\[
d_{TV}(W,Z) \le \frac{2\sqrt{5} \|\Sigma\|^{3/2} \kappa_1\kappa_2}{\sigma^2}.
\]
\end{thm}
\noindent 
Note that when $X_1,\ldots,X_n$ are gaussian, we have $c_2 = 0$, and the first bound becomes simpler. 
For an elementary illustrative  application of Theorem \ref{poincare1}, consider the function
\[
g(x) = \frac{1}{\sqrt{n}}\sum_{i=1}^{n-1} x_ix_{i+1}.
\]
Then
\[
\fpar{g}{x_i} = \frac{x_{i-1} + x_{i+1}}{\sqrt{n}},
\]
with the convention that $x_0 \equiv x_{n+1} \equiv 0$. Again,
\[
\mpar{g}{x_i}{x_j} =
\begin{cases}
1/\sqrt{n} &\text{ if } |i-j| = 1,\\
0 &\text{ otherwise.}
\end{cases}
\]
It follows that $\kappa_0 = O(1/\sqrt{n})$, $\kappa_1 = O(1)$, and $\kappa_2 = O(1/\sqrt{n})$, which gives a total variation error bound of order $1/\sqrt{n}$.
Note that the usual way to prove a CLT for $n^{-1/2}\sum_{i=1}^{n-1}X_i X_{i+1}$ is via martingale arguments, but total variation bounds are not trivial to obtain along that route.
\vskip.1in
\noindent {\it Remarks.}
(i) Theorem \ref{poincare1} can be viewed as a second order analog of the gaussian Poincar\'e inequality \eqref{poincareineq}.
While the Poincar\'e inequality implies that $g(X)$ is concentrated whenever the individual coordinates have small `influence' on the outcome, Theorem \ref{poincare1} says that if in addition, the `interaction' between the coordinates is small, then $g(X)$ has gaussian behavior. The magnitude of $\|\nabla^2 g(X)\|$ is a measure of this interaction.

(ii) The smallness of $\|\nabla^2 g(X)\|$ does not seem to imply that $g(X)$ has any special structure, at least from what the author understands. In particular, it does not imply that $g(X)$ breaks up as an approximately additive function as in H\'ajek projections \cite{vanzwet84, friedrich89}. It is quite mysterious, at the present level of understanding, as to what causes the gaussianity.

(iii) A problem with Theorem \ref{poincare1} is that it does not say anything about $\sigma^2$. However, in practice, we only need to know a lower bound on $\sigma^2$ to use Theorem \ref{poincare1} for proving a CLT. Sometimes this may be a lot easier to achieve than computing the exact limiting value of $\sigma^2$. This is demonstrated in some of our examples in Section \ref{examples}.

(iv) One may wonder why we work with random variables in $\ml(c_1,c_2)$ instead of just gaussian random variables. Indeed, the main purpose of this limited generality is simply to pre-empt the question `Does your result extend to the non-gaussian case?'. However, it is more serious than that: The true rate of convergence may actually differ significantly depending on whether $X$ is gaussian or not, as demonstrated in the case of Wigner matrices in Section \ref{examples}.

(v) There is a substantial body of literature on central limit theorems for general functions of independent random variables. Some examples of available techniques are: (a) the classical method of moments, (b) the martingale approach and Skorokhod embeddings, (c) the method of Haj\'ek projections and some sophisticated extensions (e.g.\ \cite{vanzwet84}, \cite{ruschendorf85}, \cite{friedrich89}), (d) Stein's method of normal approximation (e.g.\ \cite{stein72}, \cite{stein86}, \cite{goldsteinreinert97}), and (e)~the big-blocks-small-blocks technique and its modern multidimensional versions (e.g.\ \cite{bickelbreiman83}, \cite{avrambertsimas93}). For further references --- particularly on Stein's method, which is a cornerstone of our approach --- we refer to \cite{chatterjeenormal1}. Apart from the method of moments, none of the other techniques have been used for dealing with random matrix problems.

\section{The random matrix result}\label{generalclt}
Let $n$ be a fixed positive integer and $\mi$ be a finite indexing set. Suppose that for each $1\le i,j\le n$, we have a $C^2$ map $a_{ij}: \rr^{\mi} \ra \cc$. For each $x\in \rr^\mi$, let $A(x)$ be the complex $n\times n$ matrix whose $(i,j)^{\mathrm{th}}$ element is $a_{ij}(x)$. Let 
\[
f(z) = \sum_{m=0}^\infty b_m z^m
\]
be an analytic function on the complex plane. Let $X=(X_u)_{u\in\mi}$ be a collection of independent random variables in $\ml(c_1,c_2)$ for some finite $c_1,c_2$. 
Under this very general setup, we give an explicit bound on the total variation distance between the laws of $\real \tr f(A(X))$ and a gaussian random variable with matching mean and variance (here as usual, $\real z$ and $\imag z$ denote the real and imaginary parts of a complex number $z$). 

As mentioned before, the method involves some bookkeeping, partly due to the quest for generality. The algorithm requires the user to compute a few quantities associated with the matrix model, step by step as described below. First, let
\begin{equation}\label{rs}
\begin{split}
\mathcal{R} &= \{\alpha\in \cc^\mi: {\textstyle\sum_{u\in \mi}} |\alpha_u|^2 = 1\} \ \text{ and } \\ \mathcal{S} &= \{\beta\in \cc^{n\times n} : {\textstyle \sum_{i,j=1}^n} |\beta_{ij}|^2 =1 \}.
\end{split}
\end{equation}
Next, define three functions $\gamma_0$, $\gamma_1$ and $\gamma_2$ on $\rr^\mi$ as follows.
\begin{equation}\label{gamma}
\begin{split}
\gamma_0(x) &:= \sup_{u\in \mi, \|B\|=1} \biggl|\tr\biggl(B\fpar{A}{x_u}\biggr)\biggr|,\\
\gamma_1(x)
&:= \sup_{\alpha\in \mathcal{R}, \beta \in \mathcal{S}}\biggl|\sum_{u\in \mi}\sum_{i,j=1}^n \alpha_u\beta_{ij}\fpar{a_{ij}}{x_u}\biggr|,  \text{ and }\\
\gamma_2(x)
&:= \sup_{\alpha, \alpha'\in \mathcal{R},\beta \in \mathcal{S}}\biggl|\sum_{u,v\in \mi} \sum_{i,j=1}^n \alpha_u\alpha_v'\beta_{ij}\mpar{a_{ij}}{x_u}{x_v}\biggr|.
\end{split}
\end{equation}
Define two entire functions $f_1$ and $f_2$ as 
\[
f_1(z)=\sum_{m=1}^\infty m |b_m|z^{m-1} \ \text{ and } \ f_2(z) = \sum_{m=2}^\infty m(m-1)|b_m| z^{m-2}.
\]
Let $\lambda(x) = \|A(x)\|$ and $r(x) = \mathrm{rank}(A(x))$. Usually, of course, we will just have $r(x)\equiv n$. Next, define three more functions
\begin{align*}
\eta_0(x) &= \gamma_0(x) f_1(\lambda(x)), \\
\eta_1(x) &= \gamma_1(x) f_1(\lambda(x)) \sqrt{r(x)}, \text{ and}\\
\eta_2(x) &= \gamma_2(x)f_1(\lambda(x)) \sqrt{r(x)} + \gamma_1(x)^2 f_2(\lambda(x)).
\end{align*}
Finally, define three quantities $\kappa_0$, $\kappa_1$, and $\kappa_2$ as
\begin{align*}
\kappa_0 &= (\ee(\eta_0(X)^2\eta_1(X)^2))^{1/2}, \\ \kappa_1 &= (\ee \eta_1(X)^4)^{1/4}, \ \text{ and  } \\
\kappa_2 &= (\ee \eta_2(X)^4)^{1/4}.
\end{align*} 
Let us pacify the possibly disturbed reader with the assurance that we only need {\it bounds} on $\kappa_0$, $\kappa_1$, and $\kappa_2$, as oppposed to exact computations. This turns out to be particularly easy to achieve in all our examples. We are now ready to state the theorem. 
\begin{thm}\label{matrix}
Let all notation be as above. Suppose $W= \real \tr f(A(X))$ has finite fourth moment and let $\sigma^2 = \var(W)$. Let $Z$ be a normal random variable with the same mean and variance as $W$. Then
\[
d_{TV}(W,Z) \le \frac{2\sqrt{5} (c_1c_2 \kappa_0 + c_1^3 \kappa_1\kappa_2)}{\sigma^2}.
\]
If we slightly change the setup by assuming that $X$ is a gaussian random vector with mean $0$ and covariance matrix $\Sigma$, keeping all other notation the same, then the corresponding bound is
\[
d_{TV}(W,Z) \le \frac{2\sqrt{5} \|\Sigma\|^{3/2} \kappa_1\kappa_2}{\sigma^2}.
\]
\end{thm}
\noindent {\it Remarks.}
(i) A problem with Theorem \ref{matrix} is that it does not give a formula or approximation for $\sigma^2$. However, central limit theorems can still be proven if we can only compute suitable {\it lower bounds} for $\sigma^2$. In Section \ref{examples}, we show that this is eminently possible in a variety of situations (e.g.\ Theorems \ref{wignerthm} and \ref{toeplitz}).

(ii) Although the result is stated for entire functions, the concrete error bound, combined with appropriate concentration inequalities, should make it possible to prove limit theorems for general $C^1$ functions wherever required. 

(iii) Note that the matrices need not be hermitian, and the random variables need not be symmetric around zero. However, it is a significant restriction that the $X_{ij}$'s have to belong to $\ml(c_1,c_2)$ for some finite $c_1,c_2$. In particular, they cannot be discrete.

(iv) By considering $\alpha f$ instead of $f$ for arbitrary $\alpha \in \cc$, we see that the normal approximation error bound can be computed for any linear combination of the real and imaginary parts of the trace. This allows us to prove central limit theorems for the complex statistic $\tr f(A)$ via Wold's device.

(v) It is somewhat surprising that such a general result can give useful error bounds for familiar random matrix models. 
Unfortunately, the case of random unitary and orthogonal matrices seems to be harder because of the complexity in expressing them as functions of independent random variables. This is under the scope of a future project.

\section{Applications}\label{examples}
This section is devoted to working out a number of applications of Theorem \ref{poincare1}. In all cases, we produce a total variation error bound where the variance of the linear statistic, $\sigma^2$, appears as an unknown quantity. In some of the examples (e.g.\ Wigner and Wishart matrices), the limiting value of $\sigma^2$ is known from the literature. In other cases, they are yet unknown, and the central limit theorems are proven modulo this lack of knowledge about~$\sigma^2$.

The following simple lemma turns out to be very useful for bounding $\gamma_0$, $\gamma_1$, and $\gamma_2$ in the examples. Recall the definitions of the operator norm and the Hilbert-Schmidt norm of matrices from Section \ref{general}.
\begin{lmm}\label{trivial}
Suppose $A_1,\ldots,A_n$ $(n \ge 3)$ are real or complex matrices of dimensions such that the product $A_1A_2\cdots A_n$ is defined. Then 
\begin{equation}\label{normprop}
\|A_1A_2\|_{HS}\le \min\{\|A_1\|\|A_2\|_{HS}, \|A_1\|_{HS}\|A_2\|\}.
\end{equation}
Moreover, for any $1\le i< j\le n$, 
\[
|\tr(A_1A_2\cdots A_n)| \le \|A_i\|_{HS}\|A_j\|_{HS} \prod_{k\in [n]\backslash\{i,j\}} \|A_k\|.
\]
\end{lmm}
\begin{proof}
Let $b_1,\ldots, b_n$ be the columns of $A_2$. Then
\begin{align*}
\|A_1A_2\|_{HS}^2 &= \tr(A_2^*A_1^*A_1A_2) = \sum_{i=1}^n \|A_1b_i\|^2 \\
&\le \|A_1\|^2\sum_{i=1}^n\|b_i\|^2 = \|A_1\|^2 \|A_2\|^2_{HS}.
\end{align*}
Similarly, we have $\|A_1A_2\| \le \|A_1\|_{HS}\|A_2\|$.
For the other inequality, note that a simple application of the Cauchy-Schwarz inequality shows that
\[
|\tr(A_1A_2\cdots A_n)| \le \|A_1\cdots A_i\|_{HS}\|A_{i+1}\cdots A_n\|_{HS}
\]
Now by the inequality \eqref{normprop},
\[
\|A_1\cdots A_i\|_{HS} \le \|A_1\cdots A_{i-1}\| \|A_i\|_{HS}.
\]
Similarly,
\begin{align*}
\|A_{i+1}\cdots A_n\|_{HS} &\le \|A_{i+1}\cdots A_{j-1}\| \|A_j \cdots A_n\|_{HS}\\
&\le \|A_{i+1}\cdots A_{j-1}\| \|A_j\|_{HS}\|A_{j+1} \cdots A_n\|.
\end{align*}
This completes the proof.
\end{proof}

\subsection{Generalized Wigner matrices}
Suppose $X = (X_{ij})_{1\le i \le j\le n}$ is a collection of independent random variables. Let $X_{ij} = X_{ji}$ for $i>j$ and let 
\begin{equation}\label{wignermatrix}
A_n = A_n(X) := \frac{1}{\sqrt{n}}(X_{ij})_{1\le i,j\le n}.
\end{equation}
A matrix like $A_n$ is called a Wigner matrix. Central limit theorems for linear statistics of eigenvalues of Wigner matrices have been extensively studied in the literature (see e.g.\ \cite{sinaisoshnikov98, sinaisoshnikov98b, soshnikov02, andersonzeitouni06a}). While the case of gaussian entries can be dealt with using analytical techniques \cite {johansson98}, the general case requires heavy combinatorics. 
To give a flavor of the results in the literature, let us state one key theorem from \cite{sinaisoshnikov98b} (although, technically, it is not a CLT for a fixed linear statistic).
\vskip.1in
\noindent{\bf Theorem.} (Sina\u\i \;and Soshnikov \cite{sinaisoshnikov98b}, Theorem 2) {\it Let $X_{ij}$ and $A_n$ be as above. Suppose that the  $X_{ij}$'s have symmetric distributions around zero, $\ee(X_{ij}^2) = 1/4$ for all $i,j$, and there exists a constant $K$ such that for every positive integer $m$ and all $i,j$, $\ee(X_{ij}^{2m})\le (Km)^m$. Let $p_n \ra \infty$ as $n\ra \infty$ such that $p_n = o(n^{2/3})$. Then}
\[
\ee(\tr A_n^{p_n}) = 
\begin{cases}
2^{3/2} n(\pi p_n^3)^{-1/2} (1+ o(1)) & \text{\it if $p_n$ is even,}\\
0 & \text{\it if $p_n$ is odd,}
\end{cases}
\]
{\it and the distribution of $\tr A_n^{p_n} - \ee (\tr A_n^{p_n})$ converges weakly to the normal law $N(0,1/\pi)$. }
\vskip.1in
\noindent As remarked in \cite{sinaisoshnikov98b} and demonstrated in \cite{sinaisoshnikov98}, the normal approximation result can be extended to the joint distribution of the traces of various powers, and then to general analytic functions.

We wish to extend the above result to the scenario where $\ee(X_{ij}^2)$ is not the same for all $i,j$.   A wide generalization of this problem has been recently investigated by Anderson and Zeitouni \cite{andersonzeitouni06a} under the assumption that $\ee(X_{ij}^2) \sim f(\frac{i}{n}, \frac{j}{n})$ where $f$ is a continuous function on $[0,1]^2$. Under further assumptions, explicit formulas for the limiting means and variances are also obtained in \cite{andersonzeitouni06a}.

If the structural assumptions are dropped and we just assume that $\ee(X_{ij}^2)$ is bounded above and below by positive constants, then there does not  seem to be much hope of getting limiting formulas. Surprisingly, however, Theorem \ref{matrix} still allows us to prove central limit theorems.
\begin{thm}\label{wignerthm}
Let $A_n$ be the Wigner matrix defined in \eqref{wignermatrix}. Suppose that the $X_{ij}$'s are all in $\ml(c_1,c_2)$ for some finite $c_1, c_2$, and have symmetric distributions around zero. Suppose there are two positive constants $c$ and $C$ such that $c \le \ee(X_{ij}^2)\le C$ for all $i,j$. Let $p_n$ be a sequence of positive integers such that $p_n = o(\log n)$. Let $W_n = \tr(A_n^{p_n})$. Then as $n \ra \infty$,
\[
\frac{W_n - \ee(W_n)}{\sqrt{\var(W_n)}} \ \text{ converges in total variation to } \ N(0,1).
\]
Moreover, $\var(W_n)$ stays bounded away from zero. The same results are true also if $W_n = \tr f(A_n)$, where $f$ is a fixed nonzero polynomial with nonnegative coefficients.
\end{thm}
\noindent Note that the rate of growth allowed for $p_n$ is $o(\log n)$, which is significantly worse than the Sina\u\i -Soshnikov condition $p_n = o(n^{2/3})$. We do not know how to improve that at present. Neither do we know how to produce asymptotic formulas for $\ee(W_n)$ and $\var(W_n)$ as in Anderson and Zeitouni \cite{andersonzeitouni06a}. 
On the positive side, the assumption that $c\le \ee(X_{ij}^2) \le C$ is more general than any available result, as far as we know. In particular, we do not require asymptotic `continuity' of $\ee(X_{ij}^2)$ in $(i,j)$. The proof of Theorem~\ref{wignerthm} will follow from the following finite sample error bound. 
\begin{lmm}\label{wigner}
Fix $n$. Let $A = A(X)$ be the Wigner matrix defined in \eqref{wignermatrix}. Suppose the $X_{ij}$'s are in $\ml(c_1,c_2)$ for some finite $c_1,c_2$. Take an entire function $f$ and define $f_1$, $f_2$ as in Theorem~\ref{matrix}. Let $\lambda$ denote the spectral radius of $A$. Let $a = (\ee f_1(\lambda)^4)^{1/4}$ and $b = (\ee f_2(\lambda)^4)^{1/4}$. Suppose $W= \real \tr f(A)$ has finite fourth moment and let $\sigma^2 = \var(W)$. Let $Z$ be a normal random variable with the same mean and variance as $W$. Then
\[
d_{TV}(W, Z) \le \frac{2\sqrt{5}}{\sigma^2}\biggl(\frac{4c_1 c_2 a^2}{\sqrt{n}} + \frac{8c_1^3 ab}{n}\biggr).
\]
\end{lmm}
\noindent {\it Remarks.}
(i) It is well known that under mild conditions, $\lambda$ converges to a finite limit as $n \ra \infty$ (see e.g.\ \cite{bai99}, Section 2.2.1). Even exponentially decaying tail bounds are available \cite{gz00}. Thus $a$ and $b$ are generally $O(1)$ in the above bound. 

(ii) Sina\u\i \;and Soshnikov (\cite{sinaisoshnikov98}, Corollary 1) showed that $\sigma^2$ converges to a finite limit under certain conditions on~$f$ and the distribution of the $X_{ij}$'s. If these conditions are satisfied and the limit is nonzero, then we get a bound of order $1/\sqrt{n}$. Moreover, for gaussian Wigner matrices we have $c_2 = 0$ and hence a bound of order $1/n$. The difference between the gaussian and non-gaussian cases is not an accident. With $f(z) = z$, we have
\[
\tr f(A) = \frac{1}{\sqrt{n}}\sum_{i=1}^n X_{ii}.
\]
In this case we know that the error bound in the non-gaussian case is exactly of order $1/\sqrt{n}$. 
\vskip.1in 
Before proving Lemma \ref{wigner}, let us first prove Theorem \ref{wignerthm} using the lemma. The main difference between Lemma \ref{wigner} and Theorem \ref{wignerthm} is that the assumption of symmetry on the distributions of the entries allows us to compute a lower bound on the unknown quantity $\sigma^2$ and actually prove a CLT in Theorem \ref{wignerthm}. 
\begin{proof}[Proof of Theorem \ref{wignerthm}]
Let $s_{ij}^2 = \ee(X_{ij}^2)$, and let
\[
\xi_{ij} = \frac{X_{ij}}{s_{ij}}.
\]
Let $\Xi_n$ denote the matrix $\frac{1}{\sqrt{n}}(\xi_{ij})_{1\le i,j\le n}$.
Now take any collections of nonnegative integers $(\alpha_{ij})_{1\le i\le j\le n}$ and $(\beta_{ij})_{1\le i\le j\le n}$. Then
\begin{align*}
&\cov\biggl(\prod X_{ij}^{\alpha_{ij}}, \prod X_{ij}^{\beta_{ij}}\biggr) \\
&= \biggl(\prod s_{ij}^{\alpha_{ij} + \beta_{ij}}\biggr) \biggl(\prod \ee(\xi_{ij}^{\alpha_{ij} + \beta_{ij}}) - \prod \ee(\xi_{ij}^{\alpha_{ij}}) \ee (\xi_{ij}^{\beta_{ij}})\biggr),
\end{align*}
where the products are taken over $1\le i\le j\le n$. 
Now note that if $\alpha_{ij} + \beta_{ij}$ is odd, then $\ee(\xi_{ij}^{\alpha_{ij} + \beta_{ij}}) = \ee(\xi_{ij}^{\alpha_{ij}}) \ee(\xi_{ij}^{\beta_{ij}}) = 0$. If $\alpha_{ij}$ and $\beta_{ij}$ are both odd, then $\ee(\xi_{ij}^{\alpha_{ij} + \beta_{ij}}) \ge 0$ and $\ee(\xi_{ij}^{\alpha_{ij}}) = \ee(\xi_{ij}^{\beta_{ij}}) = 0$. Finally, if $\alpha_{ij}$ and $\beta_{ij}$ are both even, then 
\[
\ee(\xi_{ij}^{\alpha_{ij}}) \ee(\xi_{ij}^{\beta_{ij}}) \le (\ee(\xi_{ij}^{\alpha_{ij} + \beta_{ij}}))^{\frac{\alpha_{ij}}{\alpha_{ij} + \beta_{ij}}} (\ee(\xi_{ij}^{\alpha_{ij} + \beta_{ij}}))^{\frac{\beta_{ij}}{\alpha_{ij} + \beta_{ij}}}  = \ee(\xi_{ij}^{\alpha_{ij} + \beta_{ij}}).
\]
Thus, under all circumstances, we have
\begin{equation}\label{covnonneg}
\ee(\xi_{ij}^{\alpha_{ij} + \beta_{ij}}) \ge \ee(\xi_{ij}^{\alpha_{ij}}) \ee(\xi_{ij}^{\beta_{ij}}) \ge 0.
\end{equation}
Therefore,
\[
\cov\biggl(\prod X_{ij}^{\alpha_{ij}}, \prod X_{ij}^{\beta_{ij}}\biggr) \ge c^{\frac{1}{2}\sum(\alpha_{ij}+\beta_{ij})} \cov\biggl(\prod \xi_{ij}^{\alpha_{ij}}, \prod \xi_{ij}^{\beta_{ij}}\biggr).
\]
From this, it follows easily that for any positive integer $p_n$,
\[
\var (\tr A_n^{p_n}) \ge c^{p_n} \var(\tr \Xi_n^{p_n}).
\]
Now, by \eqref{covnonneg}, 
\[
\var(\tr \Xi_n^{p_n}) \ge \frac{1}{n^{p_n}} \sum_{1\le i_1,\ldots,i_{p_n}\le n} \var(\xi_{i_1i_2}\xi_{i_2i_3}\cdots \xi_{i_{p_n}i_1}).
\]
If $i_1,\ldots,i_{p_n}$ are distinct numbers, then
\[
\var(\xi_{i_1i_2}\xi_{i_2i_3}\cdots \xi_{i_{p_n}i_1}) = \ee(\xi_{i_1i_2}^2)\ee(\xi_{i_2i_3}^2)\cdots \ee(\xi_{i_{p_n}i_1}^2) = 1.
\]
Thus,
\[
\var(\tr \Xi_n^{p_n}) \ge \frac{n(n-1)\cdots(n-p_n+1)}{n^{p_n}},
\]
and so, if $p_n$ is a sequence of integers such that $p_n = o(n^{1/2})$, then 
\begin{equation}\label{varlower}
\var (\tr A_n^{p_n}) \ge K c^{p_n},
\end{equation}
where $K$ is a positive  constant that does not vary with $n$. 


Now note that for any nonnegative integer $\alpha_{ij}$, $\ee(\xi_{ij}^{\alpha_{ij}}) \ge 0$. Thus,
\[
\ee\biggl(\prod X_{ij}^{\alpha_{ij}}\biggr) = \prod s_{ij}^{\alpha_{ij}} \ee(\xi_{ij}^{\alpha_{ij}}) \le C^{\frac{1}{2}\sum\alpha_{ij}} \ee\biggl(\prod \xi_{ij}^{\alpha_{ij}}\biggr).
\]
In particular, for any positive integer $l$,
\[
\ee( \tr A^{l}) \le C^{l/2} \ee(\tr \Xi_n^{l}).
\]
Let $\lambda_n$ denote the spectral radius of $A_n$. Then for any positive integer $m$ and any positive even integer $l$,
\begin{align*}
\ee(\lambda_n^{m}) &\le \bigl(\ee(\tr A_n^{lm})\bigr)^{1/l} \le C^{m/2}\bigl(\ee(\tr \Xi_n^{lm})\bigr)^{1/l}.
\end{align*}
Now let $l = l_n :=  2[\log n]$. If $m_n$ is a sequence of positive integers such that $m_n = o(n^{2/3}/\log n)$, it follows from 
the Sina\u\i -Soshnikov result stated above that for all $n$,
\[
\bigl(\ee(\tr \Xi_n^{l_nm_n})\bigr)^{1/l_n} \le K' 2^{m_n}n^{1/l_n} \le K 2^{m_n},
\]
where $K'$ and $K$ are constants that do not depend  on $n$. Note that we could apply the theorem because $\xi_{ij}$'s are symmetric and $\ee(\xi_{ij}^{2m}) \le (Km)^m$ for all $m$ due to the $\ml(c_1,c_2)$ assumption. The $2^{m_n}$ term arises because $\ee(\xi_{ij}^2) = 1$ instead of $1/4$ as required in the Sina\u\i -Soshnikov theorem. Combined with the previous step, this gives
\begin{equation}\label{lambdabd}
\ee(\lambda_n^{m_n}) \le K (4C)^{m_n/2}.
\end{equation}
Now let us apply Lemma \ref{wigner} to $W = \tr A_n^{p_n}$. First, let us fix $n$. We have $f(x) = x^{p_n}$, and hence $f_1(x) = p_n x^{p_n - 1}$ and $f_2(x) = p_n(p_n - 1)x^{p_n - 2}$. It follows that both $a^2$ and $ab$ are bounded by $p_n^2 (\ee(\lambda_n^{4p_n}))^{1/2}$, which according to \eqref{lambdabd}, is bounded by $Kp_n^2 (4C)^{p_n}$. On the other hand, by \eqref{varlower}, $\sigma^2$ is lower bounded by $K c^{p_n}$. Combining, and using Lemma \ref{wigner}, we get
\[
d_{TV}(W,Z) \le \frac{Kp_n^2}{\sqrt{n}} \biggl(\frac{4C}{c}\biggr)^{p_n},
\]
where $K$ is a constant depending only on $c$, $C$, $c_1$ and $c_2$, and $Z$ is a gaussian random variable with the same mean and variance as $W$. If $p_n = o(\log n)$, the bound goes to zero. 

When $W_n = \tr f(A_n)$, where $f$ is a fixed polynomial with nonnegative coefficients, the proof goes through almost verbatim, and is in fact simpler. The nonnegativity of the coefficients is required to ensure that all monomial terms are positively correlated, so that we can get a lower bound on the variance.
\end{proof}

\begin{proof}[Proof of Lemma \ref{wigner}]
Let $\mi = \{(i,j):1\le i\le j\le n\}$. Let $x = (x_{ij})_{1\le i\le j\le n}$ denote a typical element of $\rr^\mi$. For each such $x$, let $A(x) = (a_{ij}(x))_{1\le i,j\le n}$ denote the matrix whose $(i,j)^{\mathrm{th}}$ element is $n^{-1/2}x_{ij}$ if $i\le j$ and $n^{-1/2}x_{ji}$ if $i >j$.  Then the matrix $A$ considered above is simply $A(X)$, and this puts us in the setting of Theorem~\ref{matrix}. Now,
\[
\fpar{a_{ij}}{x_{kl}} = 
\begin{cases}
n^{-1/2} &\text{ if } (i,j)=(k,l) \text{ or } (i,j) = (l,k), \\
0 &\text{ otherwise.}
\end{cases}
\]
Therefore, for any matrix $B$ with $\|B\|=1$, and $1\le k\ne l\le n$,
\[
\biggl|\tr\biggl(B\fpar{A}{x_{kl}}\biggr)\biggr| = \biggl|\frac{b_{kl} + b_{lk}}{\sqrt{n}}\biggr| \le \frac{2}{\sqrt{n}}.
\]
It is clear that the same bound holds even if $k=l$.  Thus, 
\[
\gamma_0(x) \le \frac{2}{\sqrt{n}} \ \text{ for all } x\in \rr^\mi.
\]
Next, let $\mathcal{R}$ and $\mathcal{S}$ be as in \eqref{rs}, and take any $\alpha \in \mathcal{R}$, $\beta \in \mathcal{S}$. Then by the Cauchy-Schwarz inequality, we have 
\begin{align*}
\biggl|\sum_{(k,l)\in \mi} \sum_{i,j=1}^n \alpha_{kl}\beta_{ij} \fpar{a_{ij}}{x_{kl}}\biggr| &\le \frac{1}{\sqrt{n}} \sum_{(k,l)\in \mi} |\alpha_{kl}(\beta_{kl}+ \beta_{lk})|\le \frac{2}{\sqrt{n}}.
\end{align*}
Thus, 
\[
\gamma_1(x) \le \frac{2}{\sqrt{n}} \ \text{ for all } x\in \rr^\mi.
\]
Now, it is clear that $\gamma_2(x) \equiv 0$ and $r(x) \le n$. Thus, if we define $\eta_0$, $\eta_1$, and $\eta_2$ as in Theorem \ref{matrix}, and let $\lambda(x)$ be the spectral radius of $A(x)$, then for all $x\in \rr^\mi$ we have
\begin{align*}
\eta_0(x) &\le \frac{2f_1(\lambda(x))}{\sqrt{n}},\\
\eta_1(x) &\le 2f_1(\lambda(x)), \ \text{ and }\\
\eta_2(x) &\le \frac{4f_2(\lambda(x))}{n}.
\end{align*}
This gives 
\begin{align*}
\kappa_0 \le \frac{4(\ee f_1(\lambda)^4)^{1/2}}{\sqrt{n}}, \ \kappa_1 \le 2(\ee f_1(\lambda)^4)^{1/4}, \ \text{and} \ \kappa_2 \le \frac{4(\ee f_2(\lambda)^4)^{1/4}}{n}.
\end{align*}
Plugging these values into Theorem \ref{matrix}, we get the result.
\end{proof}

\subsection{Gaussian matrices with correlated entries}
Suppose we have a collection $X = (X_{ij})_{1\le i, j\le n}$ of jointly gaussian random variables with mean zero and $n^2 \times n^2$ covariance matrix $\Sigma$. Let $A = n^{-1/2}(X_{ij})_{1\le i,j\le n}$.  Note that $A$ may be non-symmetric. Limiting behavior of the spectrum in such matrices have been recently investigated by Anderson and Zeitouni \cite{andersonzeitouni06} under special structures on $\Sigma$. We have the following general result.
\begin{prop}\label{corrwigner}
Take an entire function $f$ and define $f_1$, $f_2$ as in Theorem~\ref{matrix}. Let $\lambda$ denote the  operator norm of $A$. Let $a = (\ee f_1(\lambda)^4)^{1/4}$ and $b = (\ee f_2(\lambda)^4)^{1/4}$. Suppose $W= \real \tr f(A)$ has finite fourth moment and let $\sigma^2 = \var(W)$. Let $Z$ be a normal random variable with the same mean and variance as $W$. Then
\[
d_{TV}(W, Z) \le \frac{2\sqrt{5}\|\Sigma\|^{3/2}ab}{\sigma^2n}.
\]
\end{prop}
\begin{proof}
The computations of $\kappa_1$ and $\kappa_2$ are  exactly the same as for Wigner matrices. The only difference is that we now apply the second part of Theorem \ref{matrix}. 
\end{proof}
Of course, the limiting behavior of $\sigma^2$ is not known, so this does not prove a central limit theorem as long as such results are not established.  The term $\|\Sigma\|^{3/2}$ can often be handled by the well-known Gershgorin bound for the operator norm:
\[
\|\Sigma\|\le \max_{1\le i,j\le n} \sum_{k,l=1}^n |\sigma_{ij,kl}|,
\]
where $\sigma_{ij,kl} = \cov(X_{ij}, X_{kl})$. The next example gives a concrete application of the above result.

\subsection{Gaussian Toeplitz matrices}
Fix a number $n$ and let $X_0,\ldots,X_{n-1}$ be independent standard gaussian random variables. Let $A_n$ be the matrix
\[
A_n := n^{-1/2}(X_{|i-j|})_{1\le i,j\le n}.
\]
This is a gaussian Toeplitz matrix, of the kind recently considered in Bryc, Dembo, and Jiang \cite{brycdembojiang06} and also in M.\ Meckes \cite{meckes06} and Bose and Sen \cite{bosesen07}. Although Toeplitz determinants have been extensively studied (see e.g.\ Basor \cite{basor05} and references therein), to the best of our knowledge, there are no existing central limit theorems for general linear statistics of eigenvalues of random Toeplitz matrices. We have the following result. 
\begin{thm}\label{toeplitz}
Consider the gaussian Toeplitz matrices defined above. Let $p_n$ be a sequence of positive integers such that $p_n = o(\log n/\log \log n)$. Let $W_n = \tr (A_n^{p_n})$. Then, as $n \ra \infty$,
\[
\frac{W_n - \ee(W_n)}{\sqrt{\var(W_n)}} \ \text{ converges in total variation to } \ N(0,1).
\]
Moreover, there exists a positive constant $C$ such that $\var(W_n) \ge (C/p_n)^{p_n}n$ for all $n$. The  central limit theorem also holds for $W_n = \tr f(A_n)$, when $f$ is a fixed nonzero polynomial with nonnegative coefficients. In that case, $\var(W_n)\ge Cn$ for some positive constant $C$ depending on $f$. 
\end{thm}
\noindent {\it Remarks.} (i) Note that the theorem is only for {\it gaussian} Toeplitz matrices. In fact, considering the function $f(x)=x$, we see that a CLT need not always hold for linear statistics of non-gaussian Toeplitz matrices.

(ii) This is an example of a matrix ensemble where nothing is known about the limiting formula for $\var(W_n)$. Theorem \ref{matrix} enables us to prove the CLT even without knowing the limit of $\var(W_n)$. As before, this is possible because we can easily get lower bounds on $\var(W_n)$.
\begin{proof}
In the notation of the previous subsection, 
\[
\sigma_{ij,kl} = 
\begin{cases}
1 &\text{ if } |i-j|=|k-l|,\\
0 &\text{ otherwise.}
\end{cases}
\]
Thus,
\[
\|\Sigma\|\le \max_{i,j} \sum_{k,l} |\sigma_{ij,kl}| \le 2n.
\]
Let $\lambda_n$ denote the spectral norm of $A_n$. Using Proposition \ref{corrwigner} and the above bound on $\|\Sigma\|$, we have
\begin{equation}\label{toepbd}
d_{TV}(W_n,Z_n) \le \frac{Cp_n^2 (\ee \lambda_n^{4p_n})^{1/2}\sqrt{n}}{\var(W_n)},
\end{equation}
where $Z_n$ is a gaussian random variable with the same mean and variance as $W_n$ and $C$ is a universal constant.

In the rest of the argument we will write $p$ instead of $p_n$ to ease notation. First, note that
\[
W_n = \tr(A_n^p) = n^{-p/2}\sum_{1\le i_1,\ldots,i_p\le n}X_{|i_1-i_2|}X_{|i_2-i_3|}\cdots X_{|i_p-i_1|}.
\]
As in the proof of Theorem \ref{wignerthm}, it is easy to verify that all terms in the above sum are positively correlated with each other, and hence, for any partition $\mathcal{D}$ of the set $\{1,\ldots,n\}^p$ into disjoint subcollections,
\begin{equation}\label{varbdd}
\var(W_n) \ge n^{-p}\sum_{D\in \mathcal{D}} \var\biggl(\sum_{(i_1,\ldots,i_p)\in D}X_{|i_1-i_2|}X_{|i_2-i_3|}\cdots X_{|i_p-i_1|}\biggr).
\end{equation}
For any collection of distinct positive integers $1\le a_1,\ldots,a_{p-1}\le \lceil n/3p\rceil $, let $D_{a_1,\ldots,a_{p-1}}$ be the set of all $1\le i_1,\ldots,i_p\le n$ such that $i_{k+1}- i_k = a_k$ for $k=1,\ldots,p-1$ and $1\le i_1\le \lceil n/3 \rceil $. Clearly, $|D_{a_1,\ldots, a_{p-1}}| = \lceil n/3\rceil$.  Again, since the $a_i$'s are distinct,
\begin{align*}
&\var\biggl(\sum_{(i_1,\ldots,i_p)\in D_{a_1,\ldots,a_{p-1}}}X_{|i_1-i_2|}X_{|i_2-i_3|}\cdots X_{|i_p-i_1|}\biggr) \\
&= |D_{a_1,\ldots,a_{p-1}}|^2 \var(X_{a_1}\cdots X_{a_{p-1}} X_{a_1+\cdots+a_{p-1}}) = |D_{a_1,\ldots,a_{p-1}}|^2 \ge \frac{n^2}{9}.
\end{align*}
Next, note that the number of ways to choose $a_1,\ldots, a_{p-1}$ satisfying the restrictions is 
\[
\lceil n/3p\rceil (\lceil n/3p \rceil - 1)\cdots (\lceil n/3p\rceil - p +2 ).
\]
Since we can assume, without loss of generality, that $n\ge 4 p^2$, the above quantity can be easily seen to be lower bounded by $(n/12p)^{p-1}$. Finally, noting that if $(a_1,\ldots,a_{p-1})\ne (a_1',\ldots,a_{p-1}')$, then $D_{a_1,\ldots,a_{p-1}}$ and $D_{a_1',\ldots,a_{p-1}'}$ are disjoint, and applying \eqref{varbdd}, we get
\begin{equation}\label{toepvarbd}
\var(W_n) \ge n^{-p} \frac{n^{p-1}}{(12p)^{p-1}} \frac{n^2}{9} \ge \frac{C^pn}{p^p},
\end{equation}
where $C$ is a positive universal constant. 

Next, let $\lambda_n$ denote the spectral norm of $A_n$. By Theorem 1 of M.~Meckes \cite{meckes06}, we know that~$\ee(\lambda_n) \le C\sqrt{\log n}$. Now, it is easy to verify that the map $(X_0,\ldots,X_{n-1})\mapsto \lambda_n$ has Lipschitz constant bounded irrespective of $n$. By standard gaussian concentration results (e.g.\ Ledoux \cite{ledoux01}, Sections 5.1-5.2), it follows that for any $k$,
\[
\ee|\lambda_n - \ee(\lambda_n)|^k \le C^{k/2} k^{k/2},
\]
where, again, $C$ is a universal constant. Combining with result for $\ee(\lambda_n)$, it follows that for any $n$ and $k$,
\[
\ee(\lambda_n^k) \le (C k\log n)^{k/2}.
\]
Thus, the term $p^2 (\ee\lambda_n^{4p})^{1/2}$ in \eqref{toepbd} is bounded by $p^2 (Cp\log n)^p$. Therefore, from \eqref{toepbd} and \eqref{toepvarbd}, it follows that
\[
d_{TV}(W_n, Z_n) \le \frac{C^p p^{2p+2} (\log n)^p}{\sqrt{n}},
\]
where $C$ is a universal constant. Clearly, if $p = o(\log n/\log \log n)$, this goes to zero. This completes the proof for $W_n = \tr(A_n^{p_n})$. When $W_n = \tr f (A_n)$, where $f$ is a fixed polynomial with nonnegative coefficients, the proof goes through exactly as above. If $f(x) = c_0 + \cdots + c_k x^k$, the nonnegativity of the coefficients ensures that $\var(W_n)\ge c_k^2\var (\tr A_n^k)$, and we can re-use the bounds computed before to show that $\var(W_n) \ge C(f)n$. The rest is similar.
\end{proof}


\subsection{Wishart matrices}
Let $n\le N$ be two positive integers, and let $X = (X_{ij})_{1\le i\le n, 1\le j\le N}$ be a collection of independent random variables in $\ml(c_1,c_2)$ for some finite $c_1,c_2$. Let
\[
A = N^{-1} X X^t.
\]
In statistical parlance, the matrix $A$ is called the {\it Wishart matrix} or {\it sample covariance matrix} corresponding to the {\it data matrix} $X$. Just as in the Wigner case, linear statistics of eigenvalues of Wishart matrices also satisfy unnormalized  central limit theorems under certain conditions. This was proved for polynomial $f$ by Jonsson \cite{jonsson82}, and for a much larger class of functions in Bai and Silverstein \cite{baisilverstein04}. A different proof was recently given by Anderson and Zeitouni \cite {andersonzeitouni06a}. We have the following error bound.
\begin{prop}\label{covprop}
Let $\lambda$ be the largest eigenvalue of $A$. 
Take any entire function $f$ and define $f_1$, $f_2$ as in Theorem~\ref{matrix}. Let $a = (\ee(f_1(\lambda)^4\lambda^2))^{1/4}$ and $b = (\ee(f_1(\lambda) + 2n^{-1/2}f_2(\lambda)\lambda)^4)^{1/4}$.
Suppose $W= \real \tr f(A)$ has finite fourth moment and let $\sigma^2 = \var(W)$. Let $Z$ be a normal random variable with the same mean and variance as $W$. Then
\[
d_{TV}(W, Z) \le \frac{8\sqrt{5}}{\sigma^2}\biggl(\frac{c_1 c_2 a^2\sqrt{n}}{N} + \frac{c_1^3 ab n}{N^{3/2}}\biggr).
\]
If we now change the setup and assume that the entries of $X$ are jointly gaussian with mean $0$ and $nN \times nN$ covariance matrix $\Sigma$, keeping all other notation the same, then the corresponding bound is
\[
d_{TV}(W, Z) \le \frac{8\sqrt{5}\|\Sigma\|^{3/2} ab n}{\sigma^2N^{3/2}}.
\]
\end{prop}
\noindent {\it Remarks.}
(i) As in the Wigner case, it is well known that under mild conditions, $\lambda = O(1)$ as $n, N \ra \infty$ with $n/N\ra c\in [0,1)$. We refer to Section 2.2.2 in the survey article \cite{bai99} for details. It follows that $a$ and $b$ are~$O(1)$.

(ii) It is shown in \cite{baisilverstein04} that in the case of independent entries, if $n/N \ra c\in (0,1)$, then $\sigma^2$ converges to a finite positive constant under fairly general conditions (an explicit formula for the limit is also available). Therefore under such conditions, the first bound above is of order~$1/\sqrt{N}$. 

(iii) We should remark that the spectrum of $XX^t$ is often studied by studying the block matrix
\[
\biggl(
\begin{array}{cc}
0 & X\\
X^t & 0
\end{array}
\biggr),
\]
because
\[
\biggl(
\begin{array}{cc}
0 & X\\
X^t & 0
\end{array}
\biggr)^2 = 
\biggl(
\begin{array}{cc}
XX^t & 0\\
0 & X^tX
\end{array}
\biggr). 
\]
Thus, in principle, we can derive Proposition \ref{covprop} using the information contained in Lemma \ref{wigner}. However, for expository purposes, we prefer carry out the explicit computations necessary for applying Theorem  \ref{matrix} without resorting to the above trick. The computations will also be helpful in dealing with the double Wishart case in the next subsection.
\begin{proof}[Proof of Proposition \ref{covprop}]
First, let us define the indexing set
\[
\mi = \{(p,q): p=1,\ldots,n, q =1,\ldots, N\}.
\]
From now on, we simply write $pq$ instead of $(p,q)$. Let $x = (x_{pq})_{pq\in \mi}$ be a typical element of $\rr^\mi$. In the following, the collection $x$ is used as a matrix, and it seems that the only way to avoid confusion is to write $X$ instead of~$x$, so we do that. Generally, there is no harm in confusing this $X$ with the collection of random variables defined at the onset.

Let $\gamma_0$, $\gamma_1$, and $\gamma_2$ be defined as in \eqref{gamma}. For each $m$ and $i$, let $e_{mi}$ be the $i^{\mathrm{th}}$ coordinate vector in $\rr^m$, i.e.\ the vector whose $i^{\mathrm{th}}$ component is $1$ and the rest are zero. Then
\[
\fpar{A}{x_{pq}} = N^{-1}( e_{np}e_{Nq}^t X^t + X e_{Nq}e_{np}^t),
\]
and 
\begin{equation}\label{der2}
\begin{split}
\mpar{A}{x_{pq}}{x_{rs}} &= N^{-1}(e_{np}e_{Nq}^t e_{Ns}e_{nr}^t + e_{nr} e_{Ns}^t e_{Nq} e_{np}^t)\\
&=
\begin{cases}
N^{-1}(e_{np}e_{nr}^t + e_{nr} e_{np}^t) &\text{ if } q=s,\\
0 &\text{ otherwise.}
\end{cases}
\end{split}
\end{equation}
Now take any $n\times n$ matrix $B$ with $\|B\| = 1$. Then for any $p,q$,
\begin{align*}
\biggl|\tr\biggl(B\fpar{A}{x_{pq}}\biggr)\biggr| &= N^{-1} |\tr(B e_{np}e_{Nq}^t X^t + B Xe_{Nq} e_{np}^t)|\\
&= N^{-1}| e_{Nq}^t X^tBe_{np} + e_{np}^t BX e_{Nq}| \\
&\le 2N^{-1}\|B\|\|X\| = 2\sqrt{\frac{\lambda}{N}}.
\end{align*}
This shows that
\[
\gamma_0 \le 2\sqrt{\frac{\lambda}{N}}.
\]
Next, let $\alpha = (\alpha_{pq})_{1\le p\le n, 1\le q\le N}$, $\alpha' = (\alpha_{pq})_{1\le p\le n, 1\le q\le N}$, and $\beta = (\beta_{ij})_{1\le i, j\le n}$  be arbitrary matrices of complex numbers such that $\|\alpha\|_{HS} = \|\alpha'\|_{HS} = \|\beta\|_{HS} = 1$. Then
\begin{align*}
&\biggl|\sum_{p=1}^n \sum_{q=1}^N \sum_{i,j=1}^n \alpha_{pq}\beta_{ij} \fpar{a_{ij}}{x_{pq}}\biggr|\\
&= N^{-1}\biggl|\sum_{p=1}^n \sum_{q=1}^N \sum_{i,j=1}^n \alpha_{pq}\beta_{ij} e_{ni}^t(e_{np}e_{Nq}^t X^t + X e_{Nq} e_{np}^t)e_{nj}\biggr|\\
&= N^{-1}\biggl|\sum_{p=1}^n \sum_{q=1}^N \sum_{j=1}^n\alpha_{pq}\beta_{pj}x_{jq} + \sum_{p=1}^n \sum_{q=1}^N \sum_{i=1}^n\alpha_{pq}\beta_{ip}x_{iq}\biggr|\\
&= N^{-1}\bigl|\tr(\alpha X^t \beta^t) + \tr(\alpha X^t \beta)\bigr|.
\end{align*}
By Lemma \ref{trivial}, we have  
\begin{align*}
|\tr(\alpha X^t \beta^t)+\tr(\alpha X^t \beta)| &\le 2\|\alpha\|_{HS}\|X\| \|\beta\|_{HS} = 2\sqrt{N\lambda}.
\end{align*}
Thus,
\[
\gamma_1 \le 2\sqrt{\frac{\lambda}{N}}.
\]
Again, by the formula \eqref{der2} for second derivatives of $A$, 
\begin{align*}
&\biggl|\sum_{p,r=1}^n\sum_{q,s=1}^N \sum_{i,j=1}^n \alpha_{pq}\alpha_{rs}'\beta_{ij}\mpar{a_{ij}}{x_{pq}}{x_{rs}}\biggr| \\
&= N^{-1}\biggl|\sum_{p,r=1}^n\sum_{q=1}^N \sum_{i,j=1}^n \alpha_{pq}\alpha_{rq}'\beta_{ij}e_{ni}^t(e_{np}e_{nr}^t + e_{nr}e_{np}^t)e_{nj}\biggr|\\
&= N^{-1}\biggl|\sum_{p,r=1}^n\sum_{q=1}^N \alpha_{pq}\alpha_{rq}'(\beta_{pr}+ \beta_{rp}) \biggr|\\
&= N^{-1}\bigl|\tr(\alpha{\alpha'}^t \beta^t) + \tr(\alpha{\alpha'}^t \beta)\bigr| \le 2N^{-1}\|\alpha\|_{HS}\|\alpha'\|_{HS} \|\beta\|_{HS}.
\end{align*}
This shows that 
\[
\gamma_2\le \frac{2}{N}.
\]
Finally, note that $\mathrm{rank}(A) \le n$. Combining the bounds we get
\begin{align*}
\eta_0 \le 2f_1(\lambda)\sqrt{\frac{\lambda}{N}}, \ \ \eta_1 \le 2f_1(\lambda)\sqrt{\frac{\lambda n}{N}}, \ \text{ and } \ \eta_2 \le \frac{2f_1(\lambda)\sqrt{n}}{N} + \frac{4f_2(\lambda)\lambda}{N}.
\end{align*}
From this, we get
\begin{align*}
\kappa_0 &\le \frac{4\sqrt{n}}{N}(\ee (f_1(\lambda)^4\lambda^2))^{1/2},\\
\kappa_1 &\le \frac{2\sqrt{n}}{\sqrt{N}}(\ee(f_1(\lambda)^4 \lambda^2))^{1/4},\ \ \text{and}\\
\kappa_2 &\le \frac{2\sqrt{n}}{N}(\ee(f_1(\lambda) + 2n^{-1/2}f_2(\lambda) \lambda)^4)^{1/4}.
\end{align*}
With the aid of Theorem \ref{matrix}, this completes the proof.
\end{proof}

\subsection{Double Wishart matrices}
Let $n\le N\le M$ be three positive integers, and let $X = (X_{ij})_{1\le i\le n, 1\le j\le N}$ be and $Y = (Y_{ij})_{1\le i\le n, 1\le j\le M}$ be two collections of independent random variables in $\ml(c_1,c_2)$ for some finite $c_1,c_2$. Let
\[
A = X X^t (YY^t)^{-1}.
\]
A matrix like $A$ is called a {\it double Wishart matrix}. Double Wishart matrices are very important in statistical theory of canonical correlations (see the discussion in Section 2.2 of \cite{johnstone06}). 

If the matrices $X$ and $Y$ had independent standard gaussian entries, the matrix $XX^t(XX^t + YY^t)^{-1}$ would be known as a {\it Jacobi} matrix. In a recent preprint, Jiang \cite{jiang06} proves the CLT for the Jacobi ensemble. We have the following result.
\begin{prop}\label{double}
Let $\lambda_x$ and $\lambda_y$ be the largest eigenvalues of $N^{-1}XX^t$ and $M^{-1}YY^t$, and let $\delta_y$ be the smallest eigenvalue of $M^{-1}YY^t$. Let $\lambda = \max\{1,\lambda_x,\lambda_y, \delta_y^{-1}\}$.  
Take any entire function $f$ and define $f_1$, $f_2$ as in Theorem~\ref{matrix}. Let $a = (\ee(f_1(\lambda)^4 \lambda^{14}))^{1/4}$ and $$b = (\ee(4f_1(\lambda)\lambda^5 + 2n^{-1/2}f_2(\lambda) \lambda^7)^4)^{1/4}.$$
Suppose $W= \real \tr f(A)$ has finite fourth moment and let $\sigma^2 = \var(W)$. Let $Z$ be a normal random variable with the same mean and variance as $W$. Then
\[
d_{TV}(W, Z) \le \frac{4\sqrt{10}}{\sigma^2}\biggl(\frac{c_1 c_2 a^2N\sqrt{n}}{M^2} + \frac{2c_1^3 ab \sqrt{N}n}{M^2}\biggr).
\]
\end{prop}
\noindent {\it Remarks.}
(i) Assume that $n$, $N$, and $M$ grow to infinity at the same rate (we refer to this as the `large dimensional limit'). From the results about the extreme eigenvalues of Wishart matrices (\cite{bai99}, Section 2.2.2), it is clear that $\lambda = O(1)$, and hence $a$, $b$ are stochastically bounded. 

(ii) There are no rigorous results about the behavior of $\sigma^2$ in the large dimensional limit, other than in the gaussian case, which has been settled in \cite{jiang06}, where it is shown that $\sigma^2$ converges to a finite limit. 

(iii) When the entries of $X$ and $Y$ are jointly gaussian and some conditions on the dimensions are satisfied, the exact joint distribution of the eigenvalues of $A$ is known (see \cite{johnstone06}, Section 2.2 for references and an interesting story). While it may be possible to derive a CLT for the gaussian case using the explicit form of this density, it is hard to see how the non-gaussian case can be handled by either the method of moments or Stieltjes transforms.

(iv) In principle, it seems something could be said using the second order freeness results of Mingo and Speicher \cite{mingospeicher06}. However, to the best of our knowledge, an explicit CLT for double Wishart matrices has not been worked out using second order freeness. 
\begin{proof}[Proof of Proposition \ref{double}]
For convenience, let $C = XX^t$ and $D = YY^t$. Note that $\|C\| = \|X\|^2 = N \lambda_x$, $\|D\|= \|Y\|^2 = M\lambda_y$ and $\|D^{-1}\| = 1/(M\delta_y)$. Let the other notation be as in the proof of Proposition \ref{covprop}. Now
\[
\fpar{A}{x_{pq}} = ( e_{np}e_{Nq}^t X^t + X e_{Nq}e_{np}^t)D^{-1}.
\]
Again, using the formula 
\[
\fpar{D^{-1}}{y_{pq}} = -D^{-1}\fpar{D}{y_{pq}} D^{-1},
\]
we have
\[
\fpar{A}{y_{pq}} = - CD^{-1}( e_{np}e_{Mq}^t Y^t + Y e_{Mq}e_{np}^t)D^{-1}.
\]
Now take any $n\times n$ matrix $B$ with $\|B\| = 1$. Then for any $p,q$,
\begin{align*}
\biggl|\tr\biggl(B\fpar{A}{x_{pq}}\biggr)\biggr| &=  |\tr(B e_{np}e_{Nq}^t X^tD^{-1} + B Xe_{Nq} e_{np}^tD^{-1})|\\
&= | e_{Nq}^t X^tD^{-1}Be_{np} + e_{np}^t D^{-1}BX e_{Nq}| \\
&\le 2\|D^{-1}\|\|X\| \\
&\le \frac{2\lambda^{3/2}\sqrt{N}}{M}.
\end{align*}
Similarly,
\begin{align*}
\biggl|\tr\biggl(B\fpar{A}{y_{pq}}\biggr)\biggr| &\le 2 \|C\|\|Y\|\|D^{-1}\|^2 \le \frac{2\lambda^{7/2}N}{M^{3/2}}.
\end{align*}
Since $\lambda \ge 1$ and $N \le M$, 
\[
\gamma_0 \le \frac{2\lambda^{7/2}\sqrt{N}}{M}.
\]
Next, let $\alpha_{n\times N}$, $\tilde{\alpha}_{n\times M}$, $\alpha'_{n\times N}$, $\tilde{\alpha}'_{n\times M}$, and $\beta_{n\times n}$ be arbitrary arrays of complex numbers such that $\|\alpha\|_{HS}^2 + \|\tilde{\alpha}\|_{HS}^2 =1$, $\|\alpha'\|_{HS}^2 + \|\tilde{\alpha}'\|_{HS}^2 =1$, and $\|\beta\|_{HS} = 1$.
Then
\begin{align*}
&\biggl|\sum_{p=1}^n \sum_{q=1}^N \sum_{i,j=1}^n \alpha_{pq}\beta_{ij} \fpar{a_{ij}}{x_{pq}}\biggr|\\
&= \biggl|\sum_{p=1}^n \sum_{q=1}^N \sum_{i,j=1}^n \alpha_{pq}\beta_{ij} e_{ni}^t(e_{np}e_{Nq}^t X^t + X e_{Nq} e_{np}^t)D^{-1}e_{nj}\biggr|\\
&= \biggl|\sum_{p=1}^n \sum_{q=1}^N \sum_{j=1}^n\alpha_{pq}\beta_{pj}(X^tD^{-1})_{qj} + \sum_{p=1}^n \sum_{q=1}^N \sum_{i,j=1}^n\alpha_{pq}\beta_{ij}x_{iq}(D^{-1})_{pj}\biggr|\\
&= \bigl|\tr(\alpha X^tD^{-1} \beta^t) + \tr(D^{-1}\beta^t X \alpha^t)\bigr|\\
&\le 2\|\alpha\|_{HS} \|\beta\|_{HS}\|X\|\|D^{-1}\| \le \frac{2\|\alpha\|_{HS} \lambda^{3/2}\sqrt{N}}{M}.
\end{align*}
Similarly,
\begin{align*}
&\biggl|\sum_{p=1}^n \sum_{q=1}^M \sum_{i,j=1}^n \tilde{\alpha}_{pq}\beta_{ij} \fpar{a_{ij}}{y_{pq}}\biggr|\\
&= \biggl|\sum_{p=1}^n \sum_{q=1}^M \sum_{i,j=1}^n \tilde{\alpha}_{pq}\beta_{ij} e_{ni}^tCD^{-1}(e_{np}e_{Mq}^t Y^t + Y e_{Mq} e_{np}^t)D^{-1}e_{nj}\biggr|\\
&= \biggl|\sum_{p=1}^n \sum_{q=1}^M \sum_{i,j=1}^n\biggl(\tilde{\alpha}_{pq}\beta_{ij}(CD^{-1})_{ip}(Y^t D^{-1})_{qj} + \tilde{\alpha}_{pq}\beta_{ij}(CD^{-1}Y)_{iq}(D^{-1})_{pj}\biggr)\biggr|\\
&= \bigl|\tr(\tilde{\alpha} Y^tD^{-1}\beta^t CD^{-1}) + \tr(CD^{-1}Y\tilde{\alpha}^t D^{-1} \beta^t)\bigr|\\
&\le 2\|\tilde{\alpha}\|_{HS} \|\beta\|_{HS}\|Y\|\|C\|\|D^{-1}\|^2 \le \frac{2\|\tilde{\alpha}\|_{HS} \lambda^{7/2}N}{M^{3/2}}.
\end{align*}
Combining, and using the inequality 
\[
\|\alpha\|_{HS} + \|\tilde{\alpha}\|_{HS} \le \sqrt{2(\|\alpha\|_{HS}^2 + \|\tilde{\alpha}\|_{HS}^2)} = \sqrt{2},
\]
we get 
\[
\gamma_1 \le \frac{2\sqrt{2}\lambda^{7/2}\sqrt{N}}{M}.
\]
Next, let us compute the second derivatives. First, note that
\begin{align*}
\mpar{A}{x_{pq}}{x_{rs}} &= (e_{np}e_{Nq}^t e_{Ns}e_{nr}^t + e_{nr} e_{Ns}^t e_{Nq} e_{np}^t) D^{-1}\\
&= 
\begin{cases}
(e_{np}e_{nr}^t + e_{nr} e_{np}^t)D^{-1} &\text{ if } q = s,\\
0 &\text{ otherwise.}
\end{cases}
\end{align*}
Using Lemma \ref{trivial} in the last step below, we get
\begin{align*}
&\biggl|\sum_{p,r=1}^n\sum_{q,s=1}^N \sum_{i,j=1}^n \alpha_{pq}\alpha_{rs}'\beta_{ij}\mpar{a_{ij}}{x_{pq}}{x_{rs}}\biggr| \\
&= \biggl|\sum_{p,r=1}^n\sum_{q=1}^N \sum_{i,j=1}^n \alpha_{pq}\alpha_{rq}'\beta_{ij}e_{ni}^t(e_{np}e_{nr}^t + e_{nr}e_{np}^t)D^{-1}e_{nj}\biggr|\\
&= \biggl|\sum_{p,r=1}^n\sum_{q=1}^N\sum_{j=1}^n  \alpha_{pq}\alpha_{rq}'\beta_{pj}(D^{-1})_{rj}+ \sum_{p,r=1}^n \sum_{q=1}^N \sum_{j=1}^n \alpha_{pq}\alpha_{rq}' \beta_{rj}(D^{-1})_{pj}) \biggr|\\
&= \bigl|\tr(\alpha{\alpha'}^t D^{-1}\beta^t) + \tr(\alpha{\alpha'}^t \beta(D^{-1})^t)\bigr| \\
&\le 2\|\alpha\|_{HS}\|\alpha'\|_{HS} \|\beta\|_{HS}\|D^{-1}\|.
\end{align*}
Thus, we have
\begin{equation}\label{dxx}
\biggl|\sum_{p,r=1}^n\sum_{q,s=1}^N \sum_{i,j=1}^n \alpha_{pq}\alpha_{rs}\beta_{ij}\mpar{a_{ij}}{x_{pq}}{x_{rs}}\biggr| \le \frac{2\|\alpha\|_{HS}\|\alpha'\|_{HS}\lambda}{M}.
\end{equation}
Next, note that
\begin{align*}
\mpar{A}{x_{pq}}{y_{rs}} &= -( e_{np}e_{Nq}^t X^t + X e_{Nq}e_{np}^t)D^{-1}( e_{nr}e_{Ms}^t Y^t + Y e_{Ms}e_{nr}^t)D^{-1}.
\end{align*}
When we open up the brackets in the above expression, we get four terms. Let us deal with the first term:
\begin{align*}
&\biggl|\sum_{p,r=1}^n\sum_{q=1}^N\sum_{s=1}^M \sum_{i,j=1}^n \alpha_{pq}\tilde{\alpha}_{rs}'\beta_{ij}e_{ni}^t (e_{np}e_{Nq}^t X^tD^{-1}  e_{nr}e_{Ms}^t Y^t D^{-1} )e_{nj}\biggr|\\
&= \biggl|\sum_{p,r=1}^n\sum_{q=1}^N \sum_{s=1}^M\sum_{j=1}^n \alpha_{pq}\tilde{\alpha}_{rs}'\beta_{pj}( X^tD^{-1})_{qr} (Y^t D^{-1})_{sj}\biggr|\\
&= \bigl|\tr(\alpha X^t D^{-1} \tilde{\alpha}' Y^t D^{-1} \beta)\bigr|\le \|\alpha\|_{HS}\|\tilde{\alpha}'\|_{HS} \frac{\lambda^3\sqrt{N}}{M^{3/2}}.
\end{align*}
It can be similarly verified that the same bound holds for the other three terms as well. Combining, we get
\begin{equation}\label{dxy}
\biggl|\sum_{p,r=1}^n\sum_{q=1}^N \sum_{s=1}^M \sum_{i,j=1}^n \alpha_{pq}\tilde{\alpha}_{rs}'\beta_{ij}\mpar{a_{ij}}{x_{pq}}{y_{rs}}\biggr| \le 4\|\alpha\|_{HS} \|\tilde{\alpha}'\|_{HS} \frac{\lambda^3\sqrt{N}}{M^{3/2}}.
\end{equation}
Finally, note that
\begin{align*}
\mpar{A}{y_{pq}}{y_{rs}} &= C D^{-1}( e_{np}e_{Nq}^t Y^t + Y e_{Nq}e_{np}^t)D^{-1} ( e_{nr}e_{Ns}^t Y^t + Y e_{Ns}e_{nr}^t)D^{-1}\\
&\quad + C D^{-1}( e_{nr}e_{Ns}^t Y^t + Y e_{Ns}e_{nr}^t) D^{-1}( e_{np}e_{Nq}^t Y^t + Y e_{Nq}e_{np}^t)D^{-1}\\
&\quad - CD^{-1}(e_{np}e_{nr}^t + e_{nr} e_{np}^t)D^{-1}\ii_{\{q=s\}}.
\end{align*}
Proceeding exactly as before, it is quite easy to get the following bound.
It seems reasonable to omit the details.
\begin{equation}\label{dyy}
\biggl|\sum_{p,r=1}^n\sum_{q,s=1}^M \sum_{i,j=1}^n \tilde{\alpha}_{pq}\tilde{\alpha}_{rs}'\beta_{ij}\mpar{a_{ij}}{y_{pq}}{y_{rs}}\biggr| \le 10\|\tilde{\alpha}\|_{HS}\|\tilde{\alpha}'\|_{HS} \frac{\lambda^5N}{M^2}.
\end{equation}
Combining \eqref{dxx}, \eqref{dxy}, and \eqref{dyy},  and noting that $N \le M$, $\lambda \ge 1$, and the HS-norms of $\alpha$, $\alpha'$, $\tilde{\alpha}$, and $\tilde{\alpha}'$ are all bounded by $1$,  it is now easy to get that
\[
\gamma_2 \le\frac{16\lambda^5}{M}.
\]
Finally, note that $\mathrm{rank}(A) \le n$. Combining everything we get
\begin{align*}
\eta_0 \le \frac{2f_1(\lambda)\lambda^{7/2}\sqrt{N}}{M}, \ \ \eta_1 \le \frac{2\sqrt{2}f_1(\lambda)\lambda^{7/2} \sqrt{Nn}}{M}, \\
\ \text{ and } \ \eta_2 \le \frac{12f_1(\lambda)\lambda^5\sqrt{n}}{M} + \frac{8f_2(\lambda)\lambda^7N}{M^2}.
\end{align*}
From this, we get
\begin{align*}
\kappa_0 &\le \frac{4N\sqrt{2n}}{M^2}(\ee (f_1(\lambda)^4\lambda^{14}))^{1/2},\\
\kappa_1 &\le \frac{2\sqrt{2Nn}}{M}(\ee(f_1(\lambda)^4 \lambda^{14}))^{1/4},\ \ \text{and}\\
\kappa_2 &\le \frac{4\sqrt{n}}{M}(\ee(4f_1(\lambda)\lambda^5 + 2n^{-1/2}f_2(\lambda) \lambda^7)^4)^{1/4}.
\end{align*}
An application of Theorem \ref{matrix} completes the proof.
\end{proof}

\section{Proofs}\label{proofs}
\subsection{Proof of Theorem \ref{poincare1}}
The following basic lemma due to Charles Stein is our connection with Stein's method. For the reader's convenience, we reproduce the proof.
\begin{lmm}\label{steins}
\textup{(Stein \cite{stein86}, page 25)}
Let $Z$ be a standard gaussian random variable. Then for any random variable $W$, we have
\[
d_{TV}(W,Z) \le \sup\{|\ee(\psi(W)W-\psi'(W))|: \|\psi'\|_\infty \le 2\}.
\]
\end{lmm}
\begin{proof}
Take any $u:\rr \ra [-1,1]$.
It can be verified that the function
\begin{align*}
\varphi(x) &= e^{x^2/2}\int_{-\infty}^x e^{-t^2/2} (u(t)-\ee u(Z)) dt \\
&= - e^{x^2/2}\int_x^\infty e^{-t^2/2} (u(t)-\ee u(Z)) dt
\end{align*}
is a solution to the equation
\[
\vp(x)-x\varphi(x) = u(x) - \ee u(Z).
\]
Thus for each $x$,
\begin{align*}
\vp(x) &=  u(x) - \ee u(Z) - xe^{x^2/2}\int_x^\infty e^{-t^2/2} (u(t) - \ee u(Z)) dt.
\end{align*}
It follows that
\begin{align*}
\sup_{x\ge 0}|\vp(x)| &\le (\sup |u(x)-\ee u(Z)|) \biggl(1 + \sup_{x\ge 0} xe^{x^2/2}\int_x^\infty e^{-t^2/2} dt\biggr)\\
&\le 2\sup |u(x)-\ee u(Z)| \le 4.
\end{align*}
It can be verified that the same bound holds for $\sup_{x\le 0} |\vp(x)|$ by replacing $x$ with $-x$. Therefore, we have
\begin{align*}
|\ee u(W) - \ee u(Z)| &= |\ee(\varphi(W)W - \vp(W))|\\
&\le \sup\{|\ee(\psi(W)W-\psi'(W))|: \|\psi'\|_\infty \le 4\}.
\end{align*}
Since 
\[
d_{TV}(W,Z) = \frac{1}{2}\sup\{|\ee u(W) - \ee u(Z)| : \|u\|_\infty \le 1\},
\]
this completes the proof.
\end{proof}
\noindent The next lemma is for technical convenience.
\begin{lmm}\label{techlmm}
It suffices to prove Theorem \ref{poincare1} under the assumption that $g$, $\nabla g$ and $\nabla^2 g$ are uniformly bounded.
\end{lmm}
\begin{proof}
Suppose we have proved Theorem \ref{poincare1} under the said assumption. Take any $g\in C^2(\rr^n)$ such that $\sigma^2$ is finite. Now, if any one among $\kappa_0$, $\kappa_1$, and $\kappa_2$ is infinite, there is nothing to prove. So let us assume that they are all finite. 


Let $h:\rr^+\ra [0,1]$ be a $C^\infty$ function such that $h(t) = 1$ when $t\le 1$ and $h(t) = 0$ when $t\ge 2$. For each $\alpha > 0$ let
\[
g_\alpha (x) = g(x)h(\alpha^{-1}\| x\|).
\]
Clearly, as $\alpha \ra \infty$,
\begin{equation}\label{gag}
\begin{split}
d_{TV}(g(X), g_\alpha(X)) &\le \pp(g(X)\ne g_\alpha (X)) \le \pp(\|X\| > \alpha) \ra 0.
\end{split}
\end{equation}
Note that for any finite $\alpha$, $g_\alpha$ and its derivatives are uniformly bounded over~$\rr^n$. Now, since $\ee g(X)^2$ is finite, $|g_\alpha(x)|\le |g(x)|$ for all $x$, and $g_\alpha$ converges to $g$ pointwise, the dominated convergence theorem gives
\[
\lim_{\alpha\ra \infty} \ee g_\alpha(X) = \ee g(X), \ \text{ and } \ \lim_{\alpha\ra \infty} \ee g_\alpha(X)^2 = \ee g(X)^2.
\]
Again, since $\ee g(X)^4$ and $\kappa_0$, $\kappa_1$ and $\kappa_2$ are all finite, the same logic shows that
\[
\lim_{\alpha \ra \infty} \kappa_i(g_\alpha) = \kappa_i(g) \ \text{ for } \ i=0,1,2.
\]
These three steps combined show that if Theorem \ref{poincare1} holds for each $g_\alpha$, it must hold for $g$ as well. This completes the proof.
\end{proof}
\noindent The following result is the main ingredient in the proof of Theorem \ref{poincare1}.
\begin{lmm}\label{gaussian}
Let $Y = (Y_1,\ldots,Y_n)$ be a vector of i.i.d.\ standard gaussian random variables. Let $f:\rr^n \ra \rr$ be an absolutely continuous function such that $W= f(Y)$ has zero mean and unit variance. Assume that $f$ and its derivatives have subexponential growth at infinity. Let $Y'$ be an independent copy of $Y$ and define the function $T :\rr^n \ra \rr$ as
\[
T(y) := \int_0^1 \frac{1}{2\st} \ee\biggl(\sum_{i=1}^n \fpar{f}{y_i}(y) \fpar{f}{y_i}(\st y + \sst Y') \biggr) dt 
\]
Let $h(w) = \ee(T(Y)| W=w)$. Then $\ee h(W) = 1$. If $Z$ is standard gaussian, then
\[
d_{TV}(W,Z) \le 2\ee|h(W) - 1|\le 2[\var(T(Y))]^{1/2},
\]
where $d_{TV}$ is the total variation distance.
\end{lmm}
\begin{proof} 
Take any $\psi:\rr \ra \rr$ so that $\psi^\prime$ exists and is bounded. Then we have
\begin{align*}
&\ee(\psi(W) W) = \ee(\psi(f(Y))f(Y) - \psi(f(Y))f(Y')) \\
&= \ee \biggl(\int_0^1 \psi(f(Y))\frac{d}{dt} f(\st Y + \sst Y') dt\biggr)\\
&= \ee\biggl(\int_0^1 \psi(f(Y)) \sum_{i=1}^n \biggl(\frac{Y_i}{2\st} - \frac{Y'_i}{2\sst}\biggr)\fpar{f}{y_i}(\st Y + \sst Y') dt\biggr).
\end{align*}
Now fix $t\in (0,1)$, and let $U_t = \st Y + \sst Y'$, and $V_t=\sst Y - \st Y'$. Then $U_t$ and $V_t$ are independent standard gaussian random vectors and $Y= \st U_t + \sst V_t$.
Taking any $i$, and using the integration-by-parts formula for the gaussian measure (in going from the second to the third line below), we get
\begin{align*}
&\ee\biggl(\psi(f(Y)) \biggl(\frac{Y_i}{2\st} - \frac{Y'_i}{2\sst}\biggr)\fpar{f}{y_i}(\st Y+ \sst Y')\biggr) \\
&= \frac{1}{2\sqrt{t(1-t)}}\ee\biggl(\psi(f(\st U_t + \sst V_t)) V_{t,i} \fpar{f}{y_i}(U_t)\biggr)\\
&= \frac{1}{2\st} \ee\biggl(\psi^\prime(f(Y)) \fpar{f}{y_i}(Y) \fpar{f}{y_i}(U_t)\biggr).
\end{align*}
Note that we need the growth condition on the derivatives of $f$ to carry out the interchange of expectation and integration and the integration-by-parts. From the above, we have
\begin{align}
\ee(\psi(W)W) &= \ee\biggl(\psi^\prime(W) \int_0^1 \frac{1}{2\st}\sum_{i=1}^n \fpar{f}{y_i}(Y) \fpar{f}{y_i}(\st Y + \sst Y') dt\biggr) \nonumber \\
&= \ee(\psi^\prime(W) T(Y)) = \ee (\psi^\prime(W) h(W)) \nonumber.
\end{align}
The assertion that $\ee(h(W))=1$ now follows by taking $\psi(w) = w$ and using the hypothesis that $\ee(W^2)=1$. Also, easily, we have the upper bound 
\begin{align*}
|\ee(\psi(W)W - \psi^\prime(W))| &= |\ee(\psi^\prime(W) (h(W) -1 ))| \\
&\le  \|\psi^\prime\|_\infty\ee|h(W)-1|.
\end{align*}
A simple application of Lemma \ref{steins} completes the proof. 
\end{proof}
\noindent Theorem \ref{poincare1} follows from the above lemma if we bound $\var(T(Y))$ using the gaussian Poincar\'e inequality, as we do below. 

\begin{proof}[Proof of Theorem \ref{poincare1}] 
First off, by Lemma \ref{techlmm}, we can assume that $g$, $\nabla g$, and $\nabla^2 g$ are uniformly bounded and hence apply Lemma \ref{gaussian} without having to check for the growth conditions at infinity. 
Let $Y_1,\ldots,Y_n$ be independent standard gaussian random variables and $\varphi_1,\ldots,\varphi_n$ be functions such that $X_i = \varphi_i(Y_i)$ and $\|\varphi_i'\|_\infty \le c_1$, $\|\varphi_i''\|_\infty \le c_2$ for each~$i$. Define a function $\varphi: \rr^n \ra \rr^n$ as
$\varphi(y_1,\ldots,y_n) := (\varphi_1(y_1),\ldots,\varphi_n(y_n))$
and let 
\[
f(y) := g(\varphi(y)).
\]
Then $W = g(X) = f(Y)$. It is not difficult to see, through centering and scaling, that it suffices to prove Theorem \ref{poincare1} under the assumptions that $\ee(W) = 0$ and $\ee(W^2)= 1$ (this is where the $\sigma^2$ appears in the error bound). 
Now define $T$ as in Lemma \ref{gaussian}:
\[
T(y) = \int_0^1 \frac{1}{2\st} \ee\biggl(\sum_{i=1}^n \fpar{f}{y_i}(y) \fpar{f}{y_i}(\st y + \sst Y')\biggr) dt,
\]
where $Y'$ is an independent copy of $Y$. Our strategy for bounding $\var(T)$ is to simply use the gaussian Poincar\'e inequality:
\[
\var(T(Y))\le \ee\|\nabla T(Y)\|^2.
\]
The boundedness of $\nabla^2 g$ ensures that we can move the derivative inside the integrals when differentiating $T$:
\begin{align*}
\fpar{T}{y_i}(y) &= \ee\int_0^1 \frac{1}{2\st} \sum_{j=1}^n \mpar{f}{y_i}{y_j}(y)\fpar{f}{y_j}(\st y + \sst Y')dt \\
&\quad + \ee\int_0^1 \frac{1}{2} \sum_{j=1}^n \fpar{f}{y_j}(y)\mpar{f}{y_i}{y_j}(\st y + \sst Y')dt.
\end{align*}
Now for each $t\in [0,1]$, let $U_t = \st Y + \sst Y'$. With several applications of Jensen's inequality and the inequality $(a+b)^2 \le 2a^2 + 2b^2$, we get
\begin{equation}\label{poinbd}
\begin{split}
\ee\|\nabla T(Y)\|^2 &\le \ee\int_0^1 \frac{1}{\st} \sum_{i=1}^n\biggl(\sum_{j=1}^n \mpar{f}{y_i}{y_j}(Y)\fpar{f}{y_j}(U_t)\biggr)^2dt \\
&\quad + \ee\int_0^1 \frac{1}{2} \sum_{i=1}^n \biggl(\sum_{j=1}^n \fpar{f}{y_j}(Y)\mpar{f}{y_i}{y_j}(U_t)\biggr)^2dt
\end{split}
\end{equation}
Now, we have
\[
\fpar{f}{y_i}(y) = \fpar{g}{x_i}(\varphi(y)) \varphi_i'(y_i).
\]
Thus, if $i\ne j$,
\[
\mpar{f}{y_i}{y_j} = \mpar{g}{x_i}{x_j}(\varphi(y)) \varphi_i'(y_i) \varphi_j'(y_j).
\]
On the other hand, 
\[
\spar{f}{y_i} = \spar{g}{x_i} (\varphi(y)) \varphi_i'(y_i)^2 + \fpar{g}{x_i}(\varphi(y)) \varphi_i''(y_i).
\]
Thus, for any $y, u\in \rr^n$,
\begin{align*}
&\sum_{i=1}^n \biggl(\sum_{j=1}^n \mpar{f}{y_i}{y_j}(y)\fpar{f}{y_j}(u) \biggr)^2\\
&\le 2\sum_{i=1}^n \biggl(\sum_{j=1}^n \mpar{g}{x_i}{x_j}(\varphi(y))\varphi_i'(y_i)\varphi_j'(y_j)\fpar{g}{x_j}(\varphi(u))\varphi_j'(u_j)\biggr)^2 \\
&\quad + 2\sum_{i=1}^n \biggl(\fpar{g}{x_i}(\varphi(y)) \varphi_i''(y_i)\fpar{g}{x_i}(\varphi(u))\varphi_i'(u_i)\biggr)^2\\
&\le 2c_1^6 \|\nabla^2 g(\varphi(y))\|^2 \|\nabla g(\varphi(u))\|^2 + 2c_1^2c_2^2\sum_{i=1}^n  \biggl(\fpar{g}{x_i}(\varphi(y))\fpar{g}{x_i}(\varphi(u))\biggr)^2.
\end{align*}
Let us now fix $t\in [0,1]$, replace $y$ by $Y$ and $u$ by $U_t$ and use the above inequality to bound the first integrand on the right hand side of \eqref{poinbd}. First, note that since $U_t$ has the same law as $Y$, 
\begin{align*}
&\ee(\|\nabla^2 g(\varphi(Y))\|^2 \|\nabla g(\varphi(U_t))\|^2)\\
&\le (\ee\|\nabla^2 g(\varphi(Y))\|^4)^{1/2}( \ee\|\nabla g(\varphi(U_t))\|^4)^{1/2}\\
&= (\ee\|\nabla^2 g(X)\|^4)^{1/2}( \ee\|\nabla g(X)\|^4)^{1/2} = \kappa_1^2\kappa_2^2.
\end{align*}
For the same reason, we also have
\begin{align*}
\sum_{i=1}^n  \ee\biggl[\biggl(\fpar{g}{x_i}(\varphi(Y))\fpar{g}{x_i}(\varphi(U_t))\biggr)^2\biggr] &\le \kappa_0^2.
\end{align*}
Combining, we get
\[
\ee\sum_{i=1}^n\biggl(\sum_{j=1}^n \mpar{f}{y_i}{y_j}(Y)\fpar{f}{y_j}(U_t)\biggr)^2 \le 2c_1^6 \kappa_1^2 \kappa_2^2 + 2c_1^2 c_2^2 \kappa_0^2.
\]
Since this does not depend on $t$, it is now easy to see that the first term on the right hand side is bounded by $4c_1^6 \kappa_1^2 \kappa_2^2 + 4c_1^2 c_2^2 \kappa_0^2$. In a very similar manner, the second term can be bounded by $c_1^6 \kappa_1^2 \kappa_2^2 + c_1^2 c_2^2 \kappa_0^2$. Combining, and applying the inequality $\sqrt{a+b} \le \sqrt{a}+\sqrt{b}$, we finish the proof of first part of the theorem.

To prove the second part, let $X = AY$, where $Y$ is a vector of independent standard gaussian random variables and $A$ is a matrix such that $\Sigma = AA^t$. Define $h:\rr^n \ra \rr$ as $h(y) = g(Ay)$. It is easy to verify that 
\[
\|\nabla h(y)\| \le \|\Sigma\|^{1/2}\|\nabla g(Ay)\|\ \text{ and } \ \|\nabla^2 h(y)\| \le \|\Sigma\| \|\nabla^2 g(Ay)\|.
\]
The rest is straightforward from the first part of the theorem applied to $h(Y)$ instead of $g(X)$, noting that for the standard gaussian distribution we have $c_1 = 1$ and $c_2 = 0$.
\end{proof}

\subsection{Proof of Theorem \ref{matrix}}
Let us begin with some bounds on matrix differentials. Inequality \eqref{normprop} from Lemma \ref{trivial} is particularly useful.
\begin{lmm}\label{bd1}
Let $A=(a_{ij})_{1\le i,j\le n}$ be an arbitrary square matrix with complex entries. Let $f(z) = \sum_{m=0}^\infty b_m z^m$ be an entire function. Define two associated entire functions $f_1$ and $f_2$ as $f_1(z)=\sum_{m=1}^\infty m |b_m|z^{m-1}$ and $f_2(z) = \sum_{m=2}^\infty m(m-1)|b_m| z^{m-2}$. Then for each $i,j$, we have
\[
\fpar{}{a_{ij}} \tr(f(A)) = (f'(A))_{ji}.
\]
This gives the bounds
\begin{align*}
&\biggl|\fpar{}{a_{ij}} \tr(f(A))\biggr| \le f_1(\|A\|) \ \text{ for each $i,j$, and } \\
&\sum_{i,j}\biggl|\fpar{}{a_{ij}} \tr(f(A))\biggr|^2 \le \mathrm{rank}(A)f_1(\|A\|)^2.
\end{align*}
Next, for each $1\le i,j,k,l\le n$, let
\[
h_{ij,kl} = \mpar{}{a_{ij}}{a_{kl}}\tr(f(A)).
\] 
Let $H$ be the $n^2 \times n^2$ matrix $(h_{ij,kl})_{1\le i,j,k,l\le n}$. Then $\|H\| \le f_2(\|A\|)$.
\end{lmm}
\begin{proof}
For each $i$, let $e_i$ be the $i^{\mathrm{th}}$ coordinate vector in $\rr^n$, i.e.\ the vector whose $i^{\mathrm{th}}$ component is $1$ and the rest are zero.
Take any integer $m\ge 1$. A simple computation gives
\[
\fpar{}{a_{ij}} \tr(A^m) = \sum_{r=0}^{m-1} \tr\biggl(A^r\fpar{A}{a_{ij}} A^{m-r-1}\biggr) = m \tr\biggl(\fpar{A}{a_{ij}} A^{m-1}\biggr).
\]
Thus, 
\[
\fpar{}{a_{ij}} \tr(f(A)) = \tr\biggl(\fpar{A}{a_{ij}} f'(A)\biggr) = \tr(e_i e_j^t f'(A)) = (f'(A))_{ji}.
\]
The first inequality follows from this, since $|(f'(A))_{ji}| \le \|f'(A)\|\le f_1(\|A\|)$. 
Next, recall that if $B$ is a square matrix and $r = \mathrm{rank}(B)$, then $\|B\|_{HS}\le \sqrt{r}\|B\|$. This holds because 
\[
\|B\|_{HS}^2  =  \sum_i \lambda_i^2,
\]
where $\lambda_i$ are the singular values of $B$, whereas $\|B\| = \max_i |\lambda_i|$. Thus, if we let $r= \mathrm{rank}(A)$, then
\begin{align*}
&\biggl(\sum_{i,j}\biggl|\fpar{}{a_{ij}} \tr(f(A))\biggr|^2 \biggr)^{1/2} = \|f'(A)\|_{HS} \le \sum_{m=1}^\infty m|b_m|\|A^{m-1}\|_{HS} \\
&\le \sqrt{r}\sum_{m=1}^\infty m|b_m|\|A^{m-1}\| \le  \sqrt{r}\sum_{m=1}^\infty m|b_m|\|A\|^{m-1} = \sqrt{r}f_1(\|A\|).
\end{align*}
This proves the first claim. Next, fix some $m\ge 2$. Another simple computation shows that
\begin{align*}
\mpar{}{a_{ij}}{a_{kl}}\tr(A^m) &= m\sum_{r=0}^{m-2}\tr\biggl(\fpar{A}{a_{ij}} A^r \fpar{A}{a_{kl}} A^{m-r-2}\biggr)\\
&= m\sum_{r=0}^{m-2}\tr(e_ie_j^t A^r e_ke_l^t A^{m-r-2}).
\end{align*}
Now let $B=(b_{ij})_{1\le i,j\le n}$  and $C= (c_{ij})_{1\le i,j\le n}$ be arbitrary arrays of complex numbers such that $\sum_{i,j} |b_{ij}|^2 = \sum_{i,j}|c_{ij}|^2 =1$. Using the above expression, we get
\[
\sum_{i,j,k,l} b_{ij}c_{kl}\mpar{}{a_{ij}}{a_{kl}}\tr(A^m) = m\sum_{r=0}^{m-2} \tr(B A^r C A^{m-r-2}).
\]
Now, by Lemma \ref{trivial}, it follows that
\[
|\tr(B A_1 C A_2)| \le \|B\|_{HS}\|C\|_{HS}\|A_1\|\|A_2\|\le \|A\|^{m-2}. 
\]
Thus,
\[
\biggl|\sum_{i,j,k,l} b_{ij} c_{kl} h_{ij,kl}\biggr| \le \sum_{m=2}^\infty m(m-1) |b_m|\|A\|^{m-2} = f_2(\|A\|).
\]
Since this holds for all $B, C$ such that $\sum_{i,j}|b_{ij}|^2 = \sum_{i,j}|c_{ij}|^2 = 1$, the proof is done.
\end{proof}

\begin{proof}[Proof of Theorem \ref{matrix}]
Let all notation be as in the statement of the theorem. For any $n\times n$ matrix $B$, let $\psi(B) = \tr f(B)$. Define the map $g : \rr^\mi \ra \cc$ as $g = \psi \circ A$, that is,
\[
g(x) = \tr f(A(x)).
\]
It is useful to recall the following basic fact for the subsequent computations: For any $k$ and any vector $x\in \cc^k$, 
\begin{equation}\label{basiceq}
\|x\| = \sup\bigl\{\bigl|{\textstyle\sum_1^k} x_iy_i \bigr|: y\in \cc^k, \|y\| = 1\bigr\}.
\end{equation}
Using this and the definition of $\gamma_1$ we get 
\begin{align*}
&\|\nabla g(x)\| = \sup_{\alpha\in \mathcal{R}} \biggl|\sum_{u\in \mi} \alpha_u \fpar{g}{x_u}(x)\biggr| \\
&=\sup_{\alpha\in \mathcal{R}}\biggl| \sum_{u\in \mi} \sum_{i,j=1}^n\alpha_u\fpar{\psi}{a_{ij}}(A(x)) \fpar{a_{ij}}{x_u}(x)\biggr| \le \gamma_1 (x) \biggl(\sum_{i,j=1}^n\biggl|\fpar{\psi}{a_{ij}}(A(x))\biggr|^2\biggr)^{1/2}. 
\end{align*}
Now suppose $f_1$ is defined as in Lemma \ref{bd1}. Applying the second bound from Lemma \ref{bd1} to the last term in the above expression, we get
\begin{equation}\label{eq1}
\|\nabla g(x)\| \le \gamma_1(x) f_1(\|A(x)\|) \sqrt{\mathrm{rank}(A(x))} = \eta_1(x).
\end{equation}
Again note that for any $u\in \mi$, by Lemma \ref{bd1} and the definition of $\gamma_0$, we have 
\begin{align*}
\biggl|\fpar{g}{x_u}(x)\biggr| &=  \biggl|\tr\biggl(f'(A)\fpar{A}{x_u}\biggr)\biggr|\le \gamma_0(x)f_1(\|A(x)\|) = \eta_0(x).
\end{align*}
Thus,
\begin{align}\label{eq2}
\sum_{u\in \mi}\biggl|\fpar{g}{x_u}(x)\biggr|^4 &\le \max_{u\in \mi} \biggl|\fpar{g}{x_u}(x)\biggr|^2 \sum_{u\in \mi} \biggl|\fpar{g}{x_u}(x)\biggr|^2 \le \eta_0(x)^2 \eta_1(x)^2.
\end{align}
Next, note that
\begin{align*}
\mpar{g}{x_u}{x_v} &= \sum_{i,j=1}^n \fpar{\psi}{a_{ij}}(A(x))\mpar{a_{ij}}{x_u}{x_v}(x) \\
& \quad + \sum_{i,j,k,l=1}^n \mpar{\psi}{a_{ij}}{a_{kl}}(A(x)) \fpar{a_{ij}}{x_u}(x) \fpar{a_{kl}}{x_v}(x).
\end{align*}
Thus, if $\nabla^2 g$ denotes the Hessian matrix of $g$, then
\begin{align*}
\|\nabla^2 g(x)\| &= \sup_{\alpha, \alpha'\in \mathcal{R}} \biggl|\sum_{u,v\in \mi} \alpha_u \alpha_v' \mpar{g}{x_u}{x_v}\biggr|\\
&\le \sup_{\alpha, \alpha'\in \mathcal{R}}\biggl|\sum_{u,v\in \mi}\sum_{i,j=1}^n\alpha_u \alpha_v' \fpar{\psi}{a_{ij}}(A(x))\mpar{a_{ij}}{x_u}{x_v}(x)\biggr| \\
& \quad + \sup_{\alpha, \alpha'\in \mathcal{R}}\biggl|\sum_{u,v\in \mi}\sum_{i,j,k,l=1}^n \alpha_u\alpha_v'\mpar{\psi}{a_{ij}}{a_{kl}}(A(x)) \fpar{a_{ij}}{x_u}(x) \fpar{a_{kl}}{x_v}(x)\biggr|.
\end{align*}
Now, by the definition of $\gamma_2(x)$ and Lemma \ref{bd1}, we have
\begin{align*}
&\sup_{\alpha,\alpha'\in \mathcal{R}}\biggl|\sum_{u,v\in \mi}\sum_{i,j=1}^n\alpha_u \alpha_v' \fpar{\psi}{a_{ij}}(A(x))\mpar{a_{ij}}{x_u}{x_v}(x)\biggr|\\
 &\le \gamma_2(x) \biggl(\sum_{i,j=1}^n\biggl|\fpar{\psi}{a_{ij}}(A(x))\biggr|^2\biggr)^{1/2}\le \gamma_2(x)f_1(\|A(x)\|) \sqrt{\mathrm{rank}(A(x))}. 
\end{align*}
For the second term, note that by the definition of the operator norm and the identity \eqref{basiceq},
\begin{align*}
&\sup_{\alpha,\alpha'\in \mathcal{R}}\biggl|\sum_{u,v\in \mi}\sum_{i,j,k,l=1}^n \alpha_u\alpha_v'\mpar{\psi}{a_{ij}}{a_{kl}}(A(x)) \fpar{a_{ij}}{x_u}(x) \fpar{a_{kl}}{x_v}(x)\biggr|\\
&\le \|\nabla^2 \psi(A(x))\| \sup_{\alpha\in \mathcal{R}} \sum_{i,j=1}^n \biggl|\sum_{u\in \mi} \alpha_u\fpar{a_{ij}}{x_u}(x)\biggr|^2 \\
&= \|\nabla^2 \psi(A(x))\| \sup_{\alpha\in \mathcal{R},\beta \in \mathcal{S}} \biggl|\sum_{i,j=1}^n\sum_{u\in \mi} \alpha_u\beta_{ij}\fpar{a_{ij}}{x_u}(x)\biggr|^2.
\end{align*}
Using the third bound from Lemma \ref{bd1} and the definition of $\gamma_1(x)$, we now get
\begin{align*}
&\sup_{\alpha,\alpha'\in \mathcal{R}}\biggl|\sum_{u,v\in \mi}\sum_{i,j,k,l=1}^n \alpha_u\alpha_v'\mpar{\psi}{a_{ij}}{a_{kl}}(A(x)) \fpar{a_{ij}}{x_u}(x) \fpar{a_{kl}}{x_v}(x)\biggr|\\
&\le f_2(\|A(x)\|) \gamma_1(x)^2 .
\end{align*}
Combining the bounds obtained in the last two steps, we have
\begin{equation}\label{eq3}
\begin{split}
\|\nabla^2 g(x)\| &\le \gamma_2(x)f_1(\|A(x)\|) \sqrt{\mathrm{rank}(A(x))} + \gamma_1^2(x) f_2(\|A(x)\|)\\
&= \eta_2(x).
\end{split}
\end{equation}
Finally, since $g$ is defined on a real domain, therefore $\nabla \real g = \real \nabla g$ and $\nabla^2 \real g = \real \nabla^2 g$. Thus, $\|\nabla \real g(x)\|\le \|\nabla g(x)\|$ and $\|\nabla^2 \real g(x)\|\le \|\nabla^2 g(x)\|$. The proof is now  completed by using \eqref{eq1}, \eqref{eq2}, and \eqref{eq3} to bound $\kappa_1$, $\kappa_0$, and $\kappa_2$ in Theorem \ref{poincare1}. The second part follows from the second part of Theorem \ref{poincare1}.
\end{proof}
\vskip.2in

\noindent {\bf Acknowledgment.} The author is grateful to Persi Diaconis for many helpful remarks and suggestions, and to the referee for suggesting a number of key improvements.

\end{document}